\documentclass[pdflatex,sn-mathphys-num]{sn-jnl}

\usepackage{graphicx}%
\usepackage{multirow}%
\usepackage{amsmath,amssymb,amsfonts}%
\usepackage{amsthm}%
\usepackage{mathrsfs}%
\usepackage[title]{appendix}%
\usepackage{xcolor}%
\usepackage{textcomp}%
\usepackage{manyfoot}%
\usepackage{booktabs}%
\usepackage{algorithm}%
\usepackage{algorithmicx}%
\usepackage{algpseudocode}%
\usepackage{listings}%

\theoremstyle{thmstyleone}%
\newtheorem{theorem}{Theorem}[section]
\newtheorem{lemma}{Lemma}[section]%
\newtheorem{corollary}{Corollary}[section]%

\theoremstyle{thmstyletwo}%

\theoremstyle{thmstylethree}%
\newtheorem{definition}{Definition}[section]%

\numberwithin{equation}{section}

\begin{document}

\title[Article Title]{Transposed Poisson structures on the $q$-analog  Virasoro-like algebras and $q$-Quantum Torus Lie algebras
}


\author*[1]{\fnm{Jie} \sur{Lin}}\email{linj022@126.com}

\author[2]{\fnm{Chengyu} \sur{Liu}}\email{2563253950@qq.com}

\author[2]{\fnm{Jingjing} \sur{Jiang}}\email{745876258@qq.com}

\affil*[1]{\orgdiv{Sino-European Institute of Aviation Engineering }, \orgname{Civil Aviation University of China},  \city{Tianjin}, \postcode{300300},  \country{China}}

\affil[2]{\orgdiv{College of Science}, \orgname{Civil Aviation University of China}, \city{Tianjin}, \postcode{300300},  \country{China}}



\abstract{We investigate the transposed Poisson structures on both the $q$-analog Virasoro-like algebra and $q$-quantum torus Lie algebra considering the cases where $q$ is generic and where $q$ is a primitive root of unity, respectively. We establish the following results: When $q$ is generic, there are no non-trivial $\frac{1}{2}$-derivations and consequently, no non-trivial transposed Poisson algebra structures exist on the $q$-analog Virasoro-like algebra. Meanwhile, the $q$-quantum torus Lie algebra does possess non-trivial $\frac{1}{2}$-derivations but lacks of a non-trivial transposed Poisson structure. When $q$ is a primitive root of unity, both the $q$-analog Virasoro-like algebra and the $q$-quantum torus Lie algebra possess non-trivial  $\frac{1}{2}$-derivations. We present the non-trivial transposed Poisson algebra structure for the $q$-analog Virasoro-like algebra. However, the $q$-quantum torus Lie algebra lacks of a non-trivial transposed Poisson structure.  }




\maketitle

\section{Introduction}\label{sec1}
Poisson algebras play an important role in numerous mathematical and physical domains, including Poisson manifolds, algebraic geometry, operads, quantization theory, quantum groups, classical mechanics, and quantum mechanics. 
 The study of all possible Poisson structures with a certain Lie or associative part is an important problem in the theory of Poisson algebras. Recently, a dual notion of the Poisson algebra (transposed Poisson algebra) by exchanging the roles of the two binary operations in the Leibniz rule defining the Poisson algebra has been introduced by Bai, etc in \cite{bib3}. 
Immediately, Ferreira, Kaygorodov and Lopatkin (see \cite{bib6}) established a connection between $\frac{1}{2}$-derivations of Lie algebras and transposed Poisson algebras, offering a method to identify all transposed Poisson structures associated with a specific Lie algebra. Utilizing this methodology, they investigated transposed Poisson structures across various Lie algebras, including the Witt algebra (see \cite{bib13}), Virasoro algebra, generalized Witt Lie algebra, Block Lie algebras and superalgebras (see \cite{bib14}), solvable Lie algebra, Lie algebra of upper triangular matrices and others (see \cite{bib12}\cite{bib16}\cite{bib15}). In \cite{bib25}, Zeng, Sun and Zhang  studied transposed Poisson structures on Virasoro-type (super)algebras. Furthermore, a series of open questions in the future study of transposed Poisson structures are listed in \cite{bib5}.

 The $q$-analog Virasoro like algebra can be regarded as a $q$-deformation of the Virasoro-like algebra introduced and studied by Arnold etc (see\cite{bib1})  when they try to generalize the Virasoro algebra . There are some papers devoted to the study of structure and representations of the $q$-analog Virasoro algebra. Jiang and Meng studied its derivation Lie algebra and the automorphism group of its derivation Lie algebra (see \cite{bib9}). Zhao and Rao constructed a class of highest weight irreducible $\mathbb Z$-graded modules, and gave a sufficient and necessary condition for such a module with finite dimensional homogeneous subspaces (see \cite{bib29}).  Zhang and Zhao (see \cite{bib28}) studied the representations of the Virasoro-like algebra and its $q$-analog. Gao constructed a class of principal vertex representations for the extended affine Lie algebras coordinatized by certain quantum tori by using the representation of $q$-analog Virasoro-like algebra in \cite{bib8}. When $q$ is generic, another $q$-analog Virasoro-like algebra is introduced by Kirkman etc in \cite{bib10}. It can be realized as the universal central extension of the inner derivation Lie algebra of the quantum torus $\mathbb{C}_{q}\left[x^{\pm 1},y^{\pm1}\right]$ (see {\cite{bib2}).
  ~The quantum torus is one of the main objects in noncommutative geometry, and plays an important role in the classification of extended affine Lie algebras.  In \cite{bib27}, Zheng and Tan studied a Lie algebra obtained by semi-direct product of $q$-analog Virasoro-like algebra and quantum torus $\mathbb{C}_{q}\left[x^{\pm 1},y^{\pm1}\right]$, which called quantum toroidal Lie algebra.
  ~  Zeng  studied the automorphism groups, derivative algebras and universal coverage of $q$-quantum torus Lie algebra with $q$ a root of unity (see \cite{bib30}) which is the derived Lie subalgebra studied by Zheng and Tan.

The structure of the present artical is as follows. In Sec.\ref{sec2}, we recall some definitions and known results for studying transposed Poisson structures. In Sec.\ref{sec3}, we  prove that  when $q$ is generic, the $q$-analog Virasoro-like algebra  admits neither nontrivial $\frac{1}{2}$-derivations nor nontrivial transposed Poisson algebra structures; When $q$ is a primitive root of unity,  we provide a comprehensive discriptions of transposed Poisson algebra structures on the $q$-analog Virasoro-like algebra. In Sec.\ref{sec4}, we proved that the $q$-quantum torus Lie algebra possess nontrivial $\frac{1}{2}$-derivations, but lacks of nontrivial transposed Poisson algebra structures regardless of whether $q$ is a root of unity or not.

\section{Preliminaries}\label{sec2}

 Throughout this paper, all algebras and vector spaces are considered over the complex field.
\begin{definition} \label{df1}\cite{bib1}
Let $L$ be a vector space equipped with two nonzero bilinear operations $\cdot$ and $\left[\cdot,\cdot\right]$. The triple $\left(L,\cdot ,\left[\cdot,\cdot\right]\right)$ is
called a transposed Poisson algebra if $\left (L,\cdot\right)$ is a commutative associative
algebra and $\left(L,\left[\cdot,\cdot\right]\right)$ is a Lie algebra that satisfies the following compatibility condition
\begin{equation}\label{eq: 2.1}
2z\cdot\left [x,y\right ]=\left[ z\cdot x,y\right ]+\left[x,z\cdot y\right],~\forall~x,y,z\in L.
\end{equation}
\end{definition}

\begin{definition} \label{df2}\cite{bib5}
Let $\left(L,\left[\cdot,\cdot\right]\right)$ be a Lie algebra. A transposed Poisson structure on $\left(L,\left[\cdot,\cdot\right]\right)$ is a commutative associative multiplication $\cdot$ on $L$ which makes $\left(L,\cdot ,\left[\cdot,\cdot\right]\right)$ a transposed Poisson algebra.
\end{definition}

The study of $\delta$-derivations of Lie algebras was initiated in the papers \cite{bib7,bib11,bib23}. Now we recall the definition of $\frac{1}{2}$-derivation of a Lie algebra.

\begin{definition} \label{A} \cite{bib6}
Let $\left(L,\left[\cdot,\cdot\right]\right)$ be a Lie algebra, $\varphi$ : $L\longrightarrow L$ be a linear map. $\varphi$  is called a $\frac{1}{2}$-derivation
if it satisfies
$$\varphi\left(\left[x,y\right]\right)=\frac{1}{2}\left(\left[\varphi\left(x\right ),y\right]+\left[x,\varphi\left(y\right)\right]\right),~\forall~x,y\in L.$$
\end{definition}

\begin{lemma}\cite{bib3} \label{Q}
Let $\left(L,\left[\cdot,\cdot\right]\right)$ be a Lie algebra and $\cdot$ a new binary (bilinear) operation on L. Then $\left(L,\cdot ,\left[\cdot,\cdot\right]\right)$ is a transposed Poisson algebra if and only if $\cdot$ is commutative and associative and for every $z\in L$ the multiplication by $z$ in $\left(L,\cdot\right)$ is a $\frac{1}{2}$-derivation of $\left(L,\left[\cdot,\cdot\right]\right)$.
\end{lemma}

The basic example of a $\frac{1}{2}$-derivation is $\alpha{\rm id},$ where $\alpha$ is a field element. Such $\frac{1}{2}$-derivations are termed trivial.

\begin{lemma} \label{C} \cite{bib4}
Let $L$ be a Lie algebra without non-trivial $\frac{1}{2}$-derivations. Then every transposed Poisson structure defined on $L$ is trivial.
\end{lemma}
Let $G$ be an abelian group, $L=\underset{g \in G}{\oplus} L_{g}$ be a $G$-graded Lie algebra. We say that a $\frac{1}{2}$-derivation $\varphi$ has degree $g$ $(\operatorname{deg}(\varphi)=g)$ if $\varphi\left(L_{h}\right) \subseteq L_{g+h}$. Let $\Delta(L)$ denote the space of $\frac{1}{2}$-derivations of $L$ and
write $\Delta_{g}(L)=\{\varphi \in \Delta(L) \mid \operatorname{deg}(\varphi)=g\}$ for $g\in G$. The following  lemma is useful in our work.

\begin{lemma} \label{B}
Let $G$ be an abelian group, $L=\underset{g \in G}{\oplus} L_{g}$ be a finitely generated G-graded Lie algebra. Then $\Delta(L)=\underset{g \in G}{\oplus}\Delta_{g}(L).$
\end{lemma}

\begin{proof}
For each element $g\in G$, we denote by $\pi_{g}:L\to L_{g}$ the canonical projection.  Let $\varphi:L\to L$ be a $\frac{1}{2}$ derivation.
For $g\in G$, define $\varphi _{g}:=\underset{h\in G}{\Sigma} \pi _{g+h}\circ\varphi \circ \pi _{h}$. Let $h,k\in G,$ since for
$x _{h}\in L_{h}$ and $x _{k}\in L_{k}$, we have

$$
\begin{aligned}
\varphi _{g}\left(\left[x_{h}, x_{k}\right]\right)
&=\underset{h\in G}{\Sigma} \pi _{g+h}\circ\varphi \circ \pi _{h} \left(\left[x_{h}, x_{k}\right]\right) \\
&=\pi _{g+h+k}\circ\varphi \left(\left[x_{h}, x_{k}\right]\right) \\
&=\frac{1}{2} \pi _{g+h+k} \left(\left[\varphi \left ( x_{h} \right ) , x_{k}\right]\right)+
\frac{1}{2} \pi _{g+h+k}\left(\left[x_{h},\varphi \left ( x_{k} \right ) \right]\right)\\
&=\frac{1}{2} \left[\pi _{g+h}\left(\varphi\left(x_{h}\right)\right), x_{k}\right]+
\frac{1}{2} \left[x_{h},\pi _{g+k}\left(\varphi\left(x_{k}\right)\right)\right]\\
&=\frac{1}{2} \left[\varphi _{g} \left(x_{h}\right), x_{k}\right]+
\frac{1}{2} \left[x_{h},\varphi _{g}\left(x_{k}\right)\right],
\end{aligned}
$$
consequently, $\varphi _{g}\in\Delta(L)$.

Since $L$ is finitely generated, there is a finite subset
$S\subset L$ generating $L$. Then there exist finite sets $Q,R\subset G$ such that
$S\subset \underset{g\in Q}{\Sigma} L_{g}$ and $\varphi\left (S\right)\subset\underset{g\in R}{\Sigma} L_{g}$.
Let $T=\left \{ g-h\mid h\in Q,g\in R \right \}$. Then $T$ is finite and we obtain, for $y\in S$
$$
\begin{aligned}
\varphi \left ( y \right ) &=\underset{g\in R}{\Sigma}\pi_{g} \circ\varphi  \left ( y \right ) \\
&=\underset{g\in R}{\Sigma} \pi _{g} \circ\varphi
\circ\left(\underset{h\in Q}{\Sigma}\pi_{h}\left(y\right)\right) \\
&=\underset{g\in R}{\Sigma} \underset{h\in Q}{\Sigma}\pi _{g} \circ\varphi
\circ \pi_{h}\left ( y \right )  \\
&=\underset{h\in Q}  {\Sigma} \left ( \underset {g\in R}{\Sigma}\pi _{g-h+h} \circ\varphi
\circ \pi_{h}\left ( y \right ) \right ) \\
&=\underset{h\in Q}{\Sigma} \left ( \underset{q\in T}{\Sigma}\pi _{q+h} \circ\varphi
\circ \pi_{h}\left ( y \right ) \right ) \\
&= \underset{q\in T}{\Sigma}\underset{h\in Q}{\Sigma} \pi _{q+h} \circ\varphi
\circ \pi_{h}\left ( y \right ) \\
&= \underset{q\in T}{\Sigma}\underset{h\in G}{\Sigma} \pi _{q+h} \circ\varphi
\circ \pi_{h}\left ( y \right ) \\
&=\underset{q\in T}{\Sigma}\varphi _{q}\left ( y \right ).
\end{aligned}
$$
This shows that the $\frac{1}{2}$ derivations $\varphi$ and $\underset{q\in T}{\Sigma}\varphi_{q}$ coincide on $S$. As $S$ generates $L$, we obtain $\varphi=\underset{q\in T}{\Sigma}\varphi_{q}$. This proves the assertion.
\end{proof}

Throughout this paper, we denote by $\mathbb{C}$, $\mathbb{Z}$, $\mathbb{N}$ the sets of all complex numbers, all integers and all positive integers, respectively,
and $\mathbb{C}^{*}=\mathbb{C}\setminus{\left \{ 0\right \} }$, $\mathbb{Z}^{*}=\mathbb{Z}\setminus{\left \{ 0\right \} }$, $\mathbb{N}^{*}=\mathbb{N}\setminus{\left \{ 0\right \} }$. We note that $\textbf{e}_{1}=\left(1,0\right)$ and $\textbf{e}_{2}=\left(0,1\right)\in \mathbb{Z} ^{2}$. Consequently, $\mathbb{Z}^{2}=\mathbb{Z}\textbf{e}_{1}\oplus \mathbb{Z}\textbf{e}_{2}$. Unless otherwise specified, we use $\textbf{m} = \left(m_{1}, m_{2}\right)$ to denote an element in $\mathbb{Z}^{2}$.

From now on we assume  all summations of the form $\sum\limits_{{\bf i}\in \mathbb Z^2}\alpha^{\bf i}L_{\bf m+i}$ are finite.
\section{Transposed Poisson structures on $q$-Virasoro-like algebras}\label{sec3}
Let $1\neq q\in \mathbb{C}^{*}$. The $q$-analog  Virasoro-like algebra is the complex Lie algebra (see\cite{bib10}) with basis $\{L(\textbf{m})\mid\textbf{m}\in\mathbb{Z}^{2}\setminus\left\{{\bf 0}\right\}\}$ and subject to the Lie bracket defined by
\begin{equation}\label{eq:44.1}
\left[L _{\textbf{m}},L _{\textbf{n}}\right]=\lambda\left(\textbf{m},\textbf{n}\right)L_{\textbf{m}+\textbf{n}},
~\forall~\textbf{m},\textbf{n}\in\mathbb{Z}^{2}\setminus\left\{{\bf 0}\right\},
\end{equation}
where $\lambda\left(\textbf{m},\textbf{n}\right)=q^{m_{2}n_{1}}-q^{m_{1}n_{2}}$. For convenience, we conventionally set $L_{\bf 0}=0$. For brevity we will call this algebra $q$-Virasoro-like algebra. In this section, we will study the transposed Poisson structures on the $q$-Virasoro-like algebra in two cases: when $q$ is not a root of unity and when $q$ is a primitive root of unity of degree $t\in\mathbb{N}^{*}\setminus{\left \{ 1,2 \right \} }$. 

\subsection{The case where $q$ is not a root of unity for $q$-Virasoro-like algebra}\label{subsec2}
In this subsection, we suppose $1\neq q\in \mathbb{C}^{*}$ is not a root of unity, i.e. $q^{t}\ne 1$ for all $t\in \mathbb{N}^*$, and denote the $q$-Virasoro-like algebra by $V_{q}$. Observe that $V_{q}= \underset{\textbf{m}\in \mathbb{Z}^{2} }{\oplus} \left ( V_{q} \right ) _{\textbf{m}}$ is a $\mathbb{Z}^{2}$-grading, where
$\left ( V_{q} \right ) _{\bf{m}}=\mathbb{C} L_{\textbf{m}},$ for $\textbf{m}\in\mathbb{Z}^{2}$. By Lemma 4.1 in \cite{bib28}, one can see that $V_{q}$ can be generated by the finite set $\left \{  L_{(\pm 1,0)}, L_{(0,\pm 1)} \right \}$.

\begin{theorem} \label{E}
$\Delta(V_{q})=span_{\mathbb{C} } \left \{ {\rm id} \right \}$.
\end{theorem}

\begin{proof}Let $\varphi $ be a $\frac{1}{2}$-derivation of $V_{q}$ and $\textbf{i}\in \mathbb{Z}^{2}$, then by Lemma \ref{B}, the $\mathbb{Z}^{2}$-grading of $V_{q}$ induces the decomposition $\varphi =\underset{\textbf{i}\in\mathbb{Z}^{2} }{\Sigma} \varphi_{\textbf{i}}$, where $\varphi_{\textbf{i}}$ is also a $\frac{1}{2}$ derivation of $V_{q}$. So we write
\begin{equation}\label{eq:3.2}
\varphi _{\textbf{i}}\left(L_{\textbf{m}}\right)=\alpha_{\textbf{m}} L_{\textbf{m+i}},~\forall~\textbf{m}\in\mathbb{Z}^{2}\setminus\left\{{\bf 0}\right\}.
\end{equation}
Based on Definition \ref{A}, we obtain
\begin{equation*}
2\varphi_{\textbf{i}}\left(\left[L _{\textbf{m}},L _{\textbf{n}}\right]\right)=\left[\varphi_{\textbf{i}}\left(L_{\textbf{m}}\right),L_{\textbf{n}}\right]+\left[L _{\textbf{m}},\varphi_{\textbf{i}}\left(L_{\textbf{n}}\right)\right],~\forall~\textbf{m},\textbf{n}\in\mathbb{Z}^{2}\setminus\left\{{\bf 0}\right\}.
\end{equation*}
Then by Eq.\eqref{eq:44.1} and Eq.\eqref{eq:3.2}, we have
\begin{equation}\label{eq:5.1}
2\lambda\left(\textbf{m},\textbf{n}\right)\alpha_{\textbf{m}+\textbf{n}}=
\alpha_{\textbf{m}}\lambda\left(\textbf{m}+\textbf{i},\textbf{n}\right)+
\alpha_{\textbf{n}}\lambda\left(\textbf{m},\textbf{n}+\textbf{i}\right),~\forall~\textbf{m},\textbf{n}\in\mathbb{Z}^{2}\setminus\left\{{\bf 0}\right\}.\end{equation}
To determine the coefficients, we need to consider the following cases.

\textbf{Case 1.} $\textbf{i}={\bf 0}$. 

From Eq.\eqref{eq:5.1}, it can be inferred that
\begin{equation}\label{eq:2.2}
\lambda\left(\textbf{m},\textbf{n}\right)\left(2\alpha_{\textbf{m}+\textbf{n}}-\alpha_{\textbf{m}}-\alpha_{\textbf{n}}\right)=0,~\forall~ \textbf{m},\textbf{n}\in\mathbb{Z}^{2}\setminus\left\{{\bf 0}\right\}.
\end{equation}
Particularly, taking $\textbf{n}= \textbf{e}_{1}$  and $\textbf{e}_{2}$ in Eq.\eqref{eq:2.2}, respectively, we have
\begin{equation}2\alpha_{\textbf{m}+\textbf{e}_{1}}-\alpha_{\textbf{m}}-\alpha_{\textbf{e}_{1}}=0,~\forall  ~\textbf{m}\in\mathbb{Z}\times\mathbb{Z}^{*}.\label{eq:9.11}\end{equation}
\begin{equation}2\alpha_{\textbf{m}+\textbf{e}_{2}}-\alpha_{\textbf{m}}-\alpha_{\textbf{e}_{2}}=0,~\forall  ~\textbf{m}\in\mathbb{Z}^{*}\times\mathbb{Z}.\label{eq:9.12}\end{equation}
Fix $m_{2}\in\mathbb{Z}^{*}$ and treat $\left( \alpha_{\textbf{m}}-\alpha_{\textbf{e}_{1}}\right)_{m_{1}\in\mathbb{Z}}$ as a geometric sequence, then by Eq.\eqref{eq:9.11}, we have
\begin{equation} \label{eq:9.13}
\alpha_{\textbf{m}}=\left( \alpha_{(1,m_{2})}-\alpha_{\textbf{e}_{1}}\right)\left(\frac{1}{2}\right)^{m_{1}-1}+\alpha_{\textbf{e}_{1}},~\forall~\bf m\in\mathbb{Z}\times \mathbb Z^*. 
\end{equation}
Fix $m_{1}\in\mathbb{Z}^{*}$ and treat $\left(\alpha_{\textbf{m}}-\alpha_{\textbf{e}_{2}}\right)_{m_{2}\in\mathbb{Z}}$ as a geometric sequence, then by Eq.\eqref{eq:9.12}, we have
\begin{equation}\label{eq:9.14}
\alpha_{\textbf{m}}=\left( \alpha_{(m_{1},0)}-\alpha_{\textbf{e}_{2}}\right)\left(\frac{1}{2}\right)^{m_{2}}+\alpha_{\textbf{e}_{2}},~\forall~\bf m\in\mathbb Z^*\times\mathbb{Z}. 
\end{equation}
Taking $m_{1}=1$ in Eq.\eqref{eq:9.14}, we have
\begin{equation} \alpha_{(1,m_{2})}=\left(\alpha_{\textbf{e}_{1}}-\alpha_{\textbf{e}_{2}}\right)\left(\frac{1}{2}\right)^{m_{2}}+\alpha_{\textbf{e}_{2}},~\forall  ~m_{2}\in\mathbb{Z}.\label{eq:9.15} \end{equation}
Substituting Eq.\eqref{eq:9.15} into Eq.\eqref{eq:9.13}, we obtain 
\begin{equation}\label{eq:9.18}
\alpha_{\textbf{m}}=\left(\alpha_{\textbf{e}_{1}}-\alpha_{\textbf{e}_{2}}\right)\left ( \left(\frac{1}{2}\right)^{m_{2}}-1 \right )\left ( \frac{1}{2}\right)^{m_{1}-1} +\alpha_{\textbf{e}_{1}},~\forall~\textbf{m}\in \mathbb{Z}\times \mathbb{Z}^{*}.
\end{equation}
Substituting Eq.\eqref{eq:9.18} into Eq.\eqref{eq:2.2}, then for those $\bf m\in\mathbb Z\times\mathbb{Z}^{*}$ and $\bf n\in\mathbb Z\times\mathbb{Z}^{*}$ such that $\bf m+\bf n\in\mathbb Z\times\mathbb{Z}^{*}$, we have 
\begin{eqnarray*}
&&\lambda\left(\textbf{m},\textbf{n}\right)\left(\alpha_{\textbf{e}_{1}}-\alpha_{\textbf{e}_{2}}\right)
(\left(\left(\frac{1}{2}\right)^{m_{2}+n_{2}}-1\right)\left(\frac{1}{2}\right)^{m_{1}+n_{1}-2}-
\left(\left(\frac{1}{2}\right)^{m_{2}}-1\right)\left(\frac{1}{2}\right)^{m_{1}-1}\\
&&-\left(\left(\frac{1}{2}\right)^{n_{2}}-1\right)\left(\frac{1}{2}\right)^{n_{1}-1})=0.
\end{eqnarray*}
By taking $\textbf{m}=\left(2,1\right) $, $\textbf{n}=\left(1,1\right) $ in the above equation, we get
\begin{equation*}\alpha_{\textbf{e}_{1}}=\alpha_{\textbf{e}_{2}}.\label{eq:9.17}\end{equation*}
Thus by Eq.\eqref{eq:9.18}, we get
\begin{equation*}
\alpha_{\textbf{m}}=\alpha_{\textbf{e}_{1}}=\alpha_{\textbf{e}_{2}},~\forall~\textbf{m}\in \mathbb{Z}\times \mathbb{Z}^{*}.
\end{equation*}
Particularly,
$$\alpha_{(m_{1},-1)}=\alpha_{\textbf{e}_{2}},\forall m_{1}\in \mathbb{Z}.$$
By setting $m_{2}=-1$ in Eq.\eqref{eq:9.12} and using the above equation, it follows that,
\begin{equation*}\alpha_{(m_{1},0)}=\alpha_{\textbf{e}_{2}}=\alpha_{\textbf{e}_{1}},~\forall~ m_{1}\in \mathbb{Z}^{*}.\end{equation*}
Thus we proved that for all $\textbf{m}\in\mathbb{Z}^{2}\setminus\left\{{\bf 0}\right\}$, $\alpha_{\textbf{m}}$ equals to a constant, by denoting this constant as $\alpha$, we get
$$\varphi _{\bf 0}=\alpha {\rm id}.$$

\textbf{Case 2.} $\textbf{i}=(i_{1},i_{2})\in \mathbb{Z}^{*}\times \left\{0\right\} $ or $\left\{0\right\} \times \mathbb{Z}^{*}$.
 
 Without loss of generality, we suppose $\textbf{i} \in \mathbb{Z}^{*}\times\left\{0\right\}$. By Eq.\eqref{eq:5.1}, it can be inferred that $~\forall~\textbf{m},\textbf{n}\in\mathbb{Z}^{2}\setminus\left\{\left(0,0\right)\right\},$
\begin{equation}\label{eq:7.1}
2\left(q^{m_{2}n_{1}}-q^{m_{1}n_{2}}\right) \alpha_{\textbf{m}+\textbf{n}}=
\alpha_{\textbf{m}}\left(q^{m_{2}n_{1}}-q^{(m_{1}+i_{1})n_{2}}\right)+\alpha_{\textbf{n}}\left(q^{m_{2}(n_{1}+i_{1})}-q^{m_{1}n_{2}}\right).
\end{equation}
By taking $\textbf{n}= \textbf{e}_{1}$  and $\textbf{e}_{2}$ in Eq.\eqref{eq:7.1}, respectively, we have

\begin{equation}\label{eq:7.2}
 2\left(1-q^{m_{2}}\right) \alpha_{\textbf{m}+\textbf{e}_{1}}=
\alpha_{\textbf{m}}\left(1-q^{m_{2}}\right)+\alpha_{\textbf{e}_{1}}\left(1-q^{m_{2}(1+i_{1})}\right),~\forall~\textbf{m}\in\mathbb{Z}^{2}\setminus\left\{{\bf 0}\right\},
\end{equation}

\begin{equation} \label{eq:7.3}
2\left(q^{m_{1}}-1\right) \alpha_{\textbf{m}+\textbf{e}_{2}}=
\alpha_{\textbf{m}}\left(q^{m_{1}+i_{1}}-1\right)+\alpha_{\textbf{e}_{2}}\left(q^{m_{1}}-q^{m_{2}i_{1}}\right),~\forall~\textbf{m}\in\mathbb{Z}^{2}\setminus\left\{{\bf 0}\right\}.\end{equation}
For a fixed $m_{2}\in\mathbb{Z}^{*}$,  treat $\left(\alpha_{\textbf{m}}+
\frac{\alpha_{\textbf{e}_{1}}\left(1-q^{m_{2}(1+i_{1})}\right)}{q^{m_{2}}-1}\right)_{m_{1}\in\mathbb{Z}}$ as a geometric sequence, then by Eq.\eqref{eq:7.2}, we have
\begin{equation}\label{eq:7.4}
\alpha_{\textbf{m}}=\left(\alpha_{(0, m_{2})}-
\frac{\alpha_{\textbf{e}_{1}}\left(1-q^{m_{2}(1+i_{1})}\right)}{1-q^{m_{2}}}\right)\left(\frac{1}{2}\right)^{m_{1}}+
\frac{\alpha_{\textbf{e}_{1}}\left(1-q^{m_{2}(1+i_{1})}\right)}{1-q^{m_{2}}},~\forall~\textbf{m}\in\mathbb{Z}\times\mathbb{Z}^{*}.\end{equation}
By taking $m_{1} = 0$ in Eq.\eqref{eq:7.3}, we get
\begin{equation*}
\alpha_{(0, m_{2})}=\frac{\alpha_{\textbf{e}_{2}}\left(1-q^{m_{2} i_{1}}\right)}{1-q^{i_{1}}},~\forall~m_{2}\in\mathbb{Z}^{*} .     
\end{equation*}
By substituting the above equation into Eq.\eqref{eq:7.4}, we have $\forall~\textbf{m}=\left ( m_{1},m_{2} \right ) \in\mathbb{Z}\times \mathbb{Z}^{*},$
\begin{equation}\label{eq:7.6}
\alpha_{\textbf{m}}=\frac{\alpha_{\textbf{e}_{2}}\left(1-q^{m_{2} i_{1}}\right)}{1-q^{i_{1}}}
\left(\frac{1}{2}\right)^{m_{1}}-
\frac{\alpha_{\textbf{e}_{1}}\left(1-q^{m_{2}(1+i_{1})}\right)}{1-q^{m_{2}}}\left(\left(\frac{1}{2}\right)^{m_{1}}-1\right).
\end{equation}
On the one hand, by taking $\textbf{m}=\left ( 1,1 \right )$ in Eq.\eqref{eq:7.6}, we get
\begin{equation}\alpha_{(1, 1)}=\frac{1}{2}\alpha_{\textbf{e}_{2}}+
\alpha_{\textbf{e}_{1}}\frac{\left(q^{(1+i_{1})}-1\right)}{2\left(q-1\right)}.\label{eq:6.1}\end{equation}
On the other hand, by taking $\textbf{m}= (2,1)$  and $\textbf{n}= (-1,0)$ in Eq.\eqref{eq:7.1}, we have
\begin{equation}2\left(q^{-1}-1\right)\alpha_{(1, 1)}=\left(q^{-1}-1\right)\alpha_{(2, 1)}+\left(q^{i_{1}-1}-1\right)\alpha_{(-1, 0)}.\label{eq:191}\end{equation}
By taking $\textbf{m}= (2,1)$ and $\textbf{m}= (-1,1)$ in Eq.\eqref{eq:7.6}, respectively, we have
\begin{equation}\alpha_{(2, 1)}=\frac{1}{4}\alpha_{\textbf{e}_{2}}+
\frac{3}{4}\frac{\left(1-q^{(1+i_{1})}\right)}{\left(1-q\right)}\alpha_{\textbf{e}_{1}}.\label{eq:192}\end{equation}
and
\begin{equation}\alpha_{(-1, 1)}=2\alpha_{\textbf{e}_{2}}-
\frac{\left(1-q^{(1+i_{1})}\right)}{\left(1-q\right)}\alpha_{\textbf{e}_{1}}.\label{eq:193}\end{equation}
By taking $\textbf{m}= (-1,0)$  and $\textbf{n}= (0,1)$ in Eq.\eqref{eq:7.1}, we have
$$\alpha_{(-1, 1)}=\frac{1}{2}\alpha_{\textbf{e}_{2}}+
\frac{\left(1-q^{(i_{1}-1)}\right)}{2\left(1-q^{-1}\right)}\alpha_{(-1,0)}.$$
By comparing the above equation with Eq.\eqref{eq:193}, we get
\begin{equation}\alpha_{(-1, 0)}=\frac{3\left(1-q^{-1}\right)}{\left(1-q^{(i_{1}-1)}\right)}\alpha_{\textbf{e}_{2}}+
\frac{2\left(q^{(1+i_{1})}-1\right) \left(1-q^{-1}\right)   }{\left(1-q\right)\left(1-q^{(i_{1}-1)}\right)}\alpha_{\textbf{e}_{1}}.\label{eq:194}\end{equation}
By substituting Eq.\eqref{eq:192} and Eq.\eqref{eq:194} into Eq.\eqref{eq:191}, we get
$$\alpha_{(1, 1)}=\frac{13}{8}\alpha_{\textbf{e}_{2}}-
\frac{5}{8}\frac{\left(1-q^{(1+i_{1})}\right)}{\left(1-q\right)}\alpha_{\textbf{e}_{1}}.$$
By comparing the above equation with Eq.\eqref{eq:6.1}, we get
\begin{equation}\alpha_{\textbf{e}_{2}}=\frac{\left(1-q^{(1+i_{1})}\right)}{\left(1-q\right)}\alpha_{\textbf{e}_{1}}.\label{eq:195}\end{equation}
Taking $\textbf{m}= (1,1)$  and $\textbf{n}= (-1,1)$ in Eq.\eqref{eq:7.1}, we have
$$2\left(q^{-1}-q\right)\alpha_{(0, 2)}=\left(q^{-1}-q^{(1+i_{1})}\right)\alpha_{(1, 1)}+\left(q^{(i_{1}-1)}-q\right)\alpha_{(-1, 1)}.$$
By substituting Eq.\eqref{eq:6.1}Eq.\eqref{eq:193} into the above equation, we get
\begin{equation}\alpha_{(0, 2)}=\frac{q^{-1}-q^{(1+i_{1})} +4\left(q^{(i_{1}-1)}-q\right) }{4(q^{-1}-q)}\alpha_{\textbf{e}_{2}}+
\frac{\left(q^{(1+i_{1})}-1\right) \left( q^{-1}-q^{(1+i_{1})} -2q^{(i_{1}-1)}+2q \right) }{4\left(q-1\right)(q^{-1}-q)}\alpha_{\textbf{e}_{1}}.\label{eq:196}\end{equation}
By taking $\textbf{m}= (0,2)$ in Eq.\eqref{eq:7.6}, we have
$$\alpha_{(0, 2)}=\left(1+q^{i_{1}}\right)\alpha_{\textbf{e}_{2}}.$$
By comparing the above equation with Eq.\eqref{eq:196}, we get\\
$\left(\frac{q^{-1}-q^{(1+i_{1})} +4\left(q^{(i_{1}-1)}-q\right) }{4(q^{-1}-q)}-1-q^{i_{1}}\right)\alpha_{\textbf{e}_{2}}+\frac{\left(q^{(1+i_{1})}-1\right) \left( q^{-1}-q^{(1+i_{1})} -2q^{(i_{1}-1)}+2q \right) }{4\left(q-1\right)(q^{-1}-q)}\alpha_{\textbf{e}_{1}}=0.$
By substituting Eq.\eqref{eq:195} into the above equation, we get
$$\left(1-q^{i_{1}+1}\right)\left(1+q^{i_{1}}\right)\left(q^{2}-1\right)\alpha_{\textbf{e}_{1}}=0.$$
Now we need consider the two subcases: $i_{1}=-1$ and $i_{1}\not=-1$

\textbf{Subcase 1.} $i_{1}\not=-1$. Since $q$ is not a root of unity, so $1-q^{i_{1}+1}\not=0$, $1+q^{i_{1}}\not=0$, $q^{2}-1\not=0$, by the above equation, we get
$$\alpha_{\textbf{e}_{1}}=0.$$
By substituting $\alpha_{\textbf{e}_{1}}=0$ into Eq.\eqref{eq:195}, we get
$$\alpha_{\textbf{e}_{2}}=0.$$
By substituting $\alpha_{\textbf{e}_{1}}=\alpha_{\textbf{e}_{2}}=0$ into Eq.\eqref{eq:7.6}, we get
\begin{equation*}\alpha_{\textbf{m}}=0,~\forall~\textbf{m}=\left ( m_{1},m_{2} \right ) \in\mathbb{Z}\times \mathbb{Z}^{*}.\end{equation*}

\textbf{Subcase 2.} $i_{1}=-1$. By Eq.\eqref{eq:195}, we get
$$\alpha_{\textbf{e}_{2}}=0.$$
By substituting $i_{1}=-1$ and $\alpha_{\textbf{e}_{2}}=0$ into Eq.\eqref{eq:7.6}, we get
$$\alpha_{\textbf{m}}=0,~\forall~\textbf{m}=\left ( m_{1},m_{2} \right ) \in\mathbb{Z}\times \mathbb{Z}^{*}.$$
In summay, in any case, we have
\begin{equation}\alpha_{\textbf{m}}=0,~\forall~\textbf{m}=\left ( m_{1},m_{2} \right ) \in\mathbb{Z}\times \mathbb{Z}^{*}\label{eq:198}.\end{equation}
By taking $m_{2}=-1$ in Eq.\eqref{eq:7.3} and  substituting $\alpha_{\textbf{e}_{2}}=0$ and Eq.\eqref{eq:198} into Eq.\eqref{eq:7.3}, we get
$$2\left(q^{m_{1}}-1\right)\alpha_{(m_{1},0)}=0,~\forall~m_1 \in \mathbb{Z}^{*}.$$
Then
$$\alpha_{\textbf{m}}=0,~\forall~\textbf{m}=\left(m_{1},m_{2}\right)\in\mathbb{Z}^{*}\times\left \{0\right\} .$$

In summary, for $\textbf{i}\in \mathbb{Z}^{*}\times\left \{ 0 \right \} $ or  $\left \{ 0 \right \} \times \mathbb{Z}^{*}$, we have for all $\textbf{m}\in\mathbb{Z}^{2}\setminus\left\{\left(0,0\right)\right\}$, $\varphi_{\textbf{i}} \left(L_{\textbf{m}}\right)=0$.

\textbf{Case 3.} $\textbf{i}=(i_{1},i_{2})\in\mathbb{Z}^{*}\times\mathbb{Z}^{*}$.

By taking $\textbf{n}= \textbf{e}_{1}$  and $\textbf{e}_{2}$ in Eq.\eqref{eq:5.1}, respectively, we have
\begin{equation}\label{eq:5.2}
2\left(1-q^{m_{2}}\right) \alpha_{\textbf{m}+\textbf{e}_{1}}=
\alpha_{\textbf{m}}\left(1-q^{(m_{2}+i_{2})}\right)+
\alpha_{\textbf{e}_{1}}\left(q^{m_{1}i_{2}}-q^{m_{2}(1+i_{1})}\right),~\forall ~\textbf{m}\in\mathbb{Z}^{2}\setminus\left\{{\bf 0}\right\},\end{equation}
\begin{equation}\label{eq:5.3}
2\left(q^{m_{1}}-1\right) \alpha_{\textbf{m}+\textbf{e}_{2}}=
\alpha_{\textbf{m}}\left(q^{\left(m_{1}+i_{1}\right)}-1\right)+
\alpha_{\textbf{e}_{2}}\left(q^{m_{1}(1+i_{2})}-q^{m_{2}i_{1}}\right),~\forall~ \textbf{m}\in\mathbb{Z}^{2}\setminus\left\{{\bf 0}\right\}.\end{equation}
By taking $m_{2}= 0$ in Eq.\eqref{eq:5.2}, we have
\begin{equation}\alpha_{(m_{1}, 0)}=\frac{\alpha_{\textbf{e}_{1}}\left(1-q^{m_{1} i_{2}}\right)}{1-q^{i_{2}}},~\forall~ m_{1}\in\mathbb{Z}^{*}.\label{eq:2.3}\end{equation}
By taking $m_{1} = 0$ in Eq.\eqref{eq:5.3}, we have
\begin{equation}\alpha_{(0, m_{2})}=\frac{\alpha_{\textbf{e}_{2}}\left(1-q^{m_{2} i_{1}}\right)}{1-q^{i_{1}}},~\forall~ m_{2}\in\mathbb{Z}^{*}.\label{eq:2.4}\end{equation}
By taking  $n_{1} = 0$, $m_{2} = 0$ in Eq\eqref{eq:5.1}, we get
\begin{equation*}2\left(1-q^{m_{1}n_{2}}\right)\alpha_{(m_{1},n_{2})}=\left(1-q^{n_{2}\left(m_{1}+i_{1}\right)}\right)\alpha_{(m_{1},0)}+\left(1-q^{m_{1} \left(n_{2}+i_{2} \right)}\right)\alpha_{(0,n_{2})},~\forall  ~m_{1},n_{2}\in\mathbb{Z}^{*}.\end{equation*}
By substituting Eq.\eqref{eq:2.3} and Eq.\eqref{eq:2.4} into the above equation, we have $\forall  ~m_{1},n_{2}\in\mathbb{Z}^{*}$
\begin{eqnarray}\label{eq:99.35}
&&2\left(1-q^{m_{1}n_{2}}\right)\alpha_{(m_{1},n_{2})}\\
&=&\left(1-q^{n_{2}\left(m_{1}+i_{1}\right)}\right)\frac{\left(1-q^{m_{1} i_{2}}\right)}{1-q^{i_{2}}}\alpha_{\textbf{e}_{1}}+\left(1-q^{m_{1} \left(n_{2}+i_{2} \right)}\right)\frac{\left(1-q^{n_{2} i_{1}}\right)}{1-q^{i_{1}}}\alpha_{\textbf{e}_{2}}.\nonumber
\end{eqnarray}
 For those $\textbf{m},\textbf{n}\in\mathbb{Z}^{*}\times\mathbb{Z}^{*}$ such that $\bf m+\bf n\in\mathbb{Z}^{*}\times\mathbb{Z}^{*}$, by substituting the above equation into Eq.\eqref{eq:5.1}, we have
\begin{eqnarray*}
&&2\lambda\left(\textbf{m},\textbf{n}\right)
\big(\frac{\left(1-q^{(m_{2}+n_{2})\left(m_{1}+n_{1}+i_{1}\right)}\right)\left(1-q^{(m_{1}+n_{1}) i_{2}}\right)}{(1-q^{i_{2}})\left(1-q^{(m_{1}+n_{1})(m_{2}+n_{2})}\right)}\alpha_{\textbf{e}_{1}}\\
&&+\frac{\left(1-q^{(m_{1}+n_{1}) \left(m_{2}+n_{2}+i_{2} \right)}\right)\left(1-q^{(m_{2}+n_{2}) i_{1}}\right)}{(1-q^{i_{1})}\left(1-q^{(m_{1}+n_{1})(m_{2}+n_{2})}\right)}\alpha_{\textbf{e}_{2}}\big)\\
&=&\lambda\left(\textbf{m}+\textbf{i},\textbf{n}\right)
\left(\frac{\left(1-q^{m_{2}\left(m_{1}+i_{1}\right)}\right)\left(1-q^{m_{1}i_{2}}\right)}{(1-q^{i_{2}})\left(1-q^{m_{1}m_{2}}\right)}\alpha_{\textbf{e}_{1}}+\frac{\left(1-q^{m_{1}\left(m_{2}+i_{2} \right)}\right)\left(1-q^{m_{2} i_{1}}\right)}{(1-q^{i_{1})}\left(1-q^{m_{1}m_{2}}\right)}\alpha_{\textbf{e}_{2}}\right)
\\&&+\lambda\left(\textbf{m},\textbf{n}+\textbf{i}\right)
\left(\frac{\left(1-q^{n_{2}\left(n_{1}+i_{1}\right)}\right)\left(1-q^{n_{1}i_{2}}\right)}{(1-q^{i_{2}})\left(1-q^{n_{1}n_{2}}\right)}\alpha_{\textbf{e}_{1}}+\frac{\left(1-q^{n_{1}\left(n_{2}+i_{2} \right)}\right)\left(1-q^{n_{2} i_{1}}\right)}{(1-q^{i_{1})}\left(1-q^{n_{1}n_{2}}\right)}\alpha_{\textbf{e}_{2}}\right).
\end{eqnarray*}
Particularly, by setting $\textbf{m}=(2i_{1},-2i_{2})$, $\textbf{n}=(-i_{1},i_{2})$ and $\textbf{m}=(-2i_{1},-i_{2})$, $\textbf{n}=(i_{1},-i_{2})$ respectively
into the above equation, we obtain
$$\left\{\begin{matrix}
(1-q^{i_{1}})\alpha_{\textbf{e}_{1}}+(1-q^{i_{2}})\alpha_{\textbf{e}_{2}}=0, \\
(1-q^{i_{1}})\alpha_{\textbf{e}_{1}}+2(1-q^{i_{2}})\alpha_{\textbf{e}_{2}}=0.
\end{matrix}\right.$$
Solving the system of equations above yields
$$\alpha_{\textbf{e}_{1}}=\alpha_{\textbf{e}_{2}}=0.$$
By substituting $\alpha_{\textbf{e}_{1}}=0$ and $\alpha_{\textbf{e}_{2}}=0$ into Eq\eqref{eq:99.35}, we get
$$\alpha_{(m_{1},n_{2})}=0,~\forall ~(m_{1},n_{2})\in\mathbb{Z}^{*}\times \mathbb{Z}^{*}.$$
i.e.
\begin{equation*}\alpha_{\textbf{m}}=0,~\forall~ \textbf{m}\in\mathbb{Z}^{*}\times \mathbb{Z}^{*}.\end{equation*}
By substituting $\alpha_{\textbf{e}_{1}}=0$ into Eq.\eqref{eq:2.3}, we get
\begin{equation*}\alpha_{(m_{1},0)}=0,~\forall~ m_{1}\in\mathbb{Z}^{*}.\end{equation*}
By substituting $\alpha_{\textbf{e}_{2}}=0$ into Eq.\eqref{eq:2.4}, we get
\begin{equation*}\alpha_{(0,m_{2})}=0,~\forall~ m_{2}\in\mathbb{Z}^{*}.\end{equation*}

In summary, for $\textbf{i}\in \mathbb{Z}^{*}\times\mathbb{Z}^{*} $, we have for all $\textbf{m}\in\mathbb{Z}^{2}\setminus\left\{{\bf 0}\right\}$, $\varphi_{\textbf{i}} \left(L_{\textbf{m}}\right)=0$.

Hence, combining the analysis of the three cases above, we obtain the desired result.
\end{proof}

Based on Lemma \ref{C} and Theorem \ref{E}, the following corollary can be derived.

\begin{corollary} \label{coro1}
There are no non-trivial transposed Poisson algebra structures defined on $V_{q}$.
\end{corollary}

\subsection{The case where $q$ is a root of unity for $q$-Virasoro-like algebra }\label{subsec3.2}
In this subsection and Section \ref{sec4}, we use the following notations. Let $t>2$ be a positive integer and $q$ be a primitive root of unity of degree $t$. Set
$t\mathbb{Z}=\left \{ tn\mid n\in \mathbb{Z} \right \}$,
$\left(t\mathbb{Z}\right)^{2}=\left(t\mathbb{Z}\right)\textbf{e}_{1}\oplus\left(t\mathbb{Z}\right)\textbf{e}_{2}$, $\Gamma_{1}=\left(t\mathbb{Z}\right)^{2}\setminus\left\{{\bf 0}\right\}$,
$\Gamma_{2}=\mathbb{Z}^{2}\setminus\left(t\mathbb{Z}\right)^{2}$, and $\Gamma=\Gamma_{1}\cup\Gamma_{2}$.

In this subsection, we denote the $q$-Virasoro-like algebra defined by Eq.\eqref{eq:44.1} as $ \tilde{V} _{q}$. We remark that for ${\bf m, n}\in \Gamma,$ $\lambda\left(\textbf{m},\textbf{n}\right)=0$
if and only if $q^{m_{2}n_{1}}-q^{m_{1}n_{2}}=1$, that is $t\mid\left(m_{2}n_{1}-m_{1}n_{2}\right)$. This implies that if $\textbf{m}+\textbf{n}\in \Gamma_{1}$ for some $\textbf{m}\in \Gamma_{2}$ and $\textbf{n}\in \Gamma_{2}$, then one has $\lambda\left(\textbf{m},\textbf{n}\right )=0$. By the Lemma 2.1 in \cite{bib29} , we can see the center of $\tilde{V} _{q}$ is the
subalgebra $\mathbb{Z} = \underset{\textbf{i}\in \Gamma_{1}}{\oplus} \mathbb{C} L_{\textbf{i}}$. Observe that $\tilde{V} _{q}= \underset{\textbf{m}\in \mathbb{Z}^{2} }{\oplus} \left ( \tilde{V} _{q}\right ) _{\textbf{m}}$ is a $\mathbb{Z}^{2}$-grading, where $\left ( \tilde{V} _{q} \right ) _{\textbf{m}}= \mathbb{C} L_{\textbf{m}}, \textbf{m}\in\mathbb{Z}^{2}$. One can easily see that for all $\textbf{i}\in \Gamma_{1}$, $L_{\textbf{i}}$ can not be generated by the set $\left \{ L(\textbf{m}) \mid \textbf{m}\in\Gamma\setminus\{{\bf i}\} \right \} $, so $\tilde{V} _{q}$ is not finitely generated.

\begin{theorem} \label{F}
Let $\varphi$ be a $\frac{1}{2}$-derivation of $ \tilde{V} _{q}$. Then
$$\forall {\bf m}\in \Gamma, \varphi(L_{\bf{m}})=\left\{\begin{matrix}
 \underset{{\bf i}\in \Gamma_{1}}{\Sigma}\alpha^{\bf i}L_{{\bf m}+{\bf i}}, & {\bf m}\in\Gamma_{2},\\
  \underset{{\bf i}\in \Gamma_{1}}{\Sigma}\alpha^{{\bf i}}_{\bf m} L_{{\bf m}+{\bf i}},&{\bf m}\in\Gamma_{1}.
\end{matrix}\right.$$
\end{theorem}

\begin{proof}Let $\varphi$ be a $\frac{1}{2}$-derivation of $ \tilde{V} _{q}$. Then the $\mathbb{Z}^{2}$-grading of $ \tilde{V} _{q}$ induces the decomposition $\varphi=\underset{\textbf{i}\in \mathbb{Z}^{2}}{\Sigma}\varphi_{\textbf{i}}$, where $\varphi_{\textbf{i}}$ is a linear map $\tilde{V} _{q}\to\tilde{V} _{q}$ such that
$\varphi_{\textbf{i}} \left(L_{\textbf{m}}\right)\subseteq L_{\textbf{m}+\textbf{i}}$ for all $\textbf{m}\in\Gamma$. Since
$\varphi$ is a $\frac{1}{2}$-derivation of $ \tilde{V} _{q}$, then $\varphi_{\textbf{i}}$ is also a $\frac{1}{2}$-derivation of $ \tilde{V} _{q}$ for all $\textbf{i}\in\mathbb{Z}^{2}$.
We write $\varphi\left(L_{\textbf{m}}\right)=\underset{\textbf{i}\in \mathbb{Z}^{2}}{\Sigma}\alpha_{\textbf{m}}^{\textbf{i}}L_{\textbf{m}+\textbf{i}}$. Based on Definition \ref{A}, 
by applying $\varphi$ to Eq.\eqref{eq:44.1},  we obtain
$$2\varphi(\left[L _{\textbf{m}},L _{\textbf{n}}\right])=\left[\varphi\left(L_{\textbf{m}}\right),L_{\textbf{n}}\right]+\left[L _{\textbf{m}},\varphi\left(L_{\textbf{n}}\right)\right],\forall~\textbf{m},\textbf{n}\in\Gamma.$$
Then we have
\begin{equation}2\lambda\left(\textbf{m},\textbf{n}\right) \alpha_{\textbf{m}+\textbf{n}}^{\textbf{i}}=
\alpha_{\textbf{m}}^{\textbf{i}}\lambda\left(\textbf{m}+\textbf{i},\textbf{n}\right)+
\alpha_{\textbf{n}}^{\textbf{i}}\lambda\left(\textbf{m},\textbf{n}+\textbf{i}\right),~\forall~\textbf{m},\textbf{n}\in\Gamma.\label{eq:4.2}\end{equation}
To determine the coefficients, we need to consider the following cases.

\textbf{Case 1.} $\textbf{i}\in \Gamma_{1}\cup\{\bf 0\}$. 

From Eq.\eqref{eq:4.2}, it can be inferred that

\begin{equation}\lambda\left(\textbf{m},\textbf{n}\right)\left(2\alpha_{\textbf{m}+ \textbf{n}}^{\textbf{i}}-\alpha_{\textbf{m}}^{\textbf{i}}-\alpha_{\textbf{n}}^{\textbf{i}}\right)=0,\forall~\textbf{m},\textbf{n}\in\Gamma.\label{eq:4.3}\end{equation}
By a similar argument as $\alpha_{\bf m}$ for ${\bf m}\in \mathbb Z^2\setminus\{{\bf 0}\}$ in Case 1 of Theorem\ref{E}, we can prove $\alpha^{\bf i}_{\bf m}$ is a constant for all ${\bf m}\in \Gamma.$

Thus, for $\textbf{i}\in \Gamma_{1}\cup\{\bf 0\}$, $\exists$ $\alpha_{}^{\textbf{i}}\in \mathbb{C} $ such that $\forall ~\textbf{m}\in \Gamma_{2} $, $\varphi_{\textbf{i}} \left(L_{\textbf{m}}\right)=\alpha^{\textbf{i}}L_{\textbf{m}+\textbf{i}}.$

\textbf{Case 2.} $\textbf{i}=(i_{1},i_{2})\in (\mathbb{Z}\setminus t\mathbb{Z})\times t\mathbb{Z}$ or $t\mathbb{Z}\times (\mathbb{Z}\setminus t\mathbb{Z})$.
 
 Without loss of generality, we suppose $\textbf{i}\in(\mathbb{Z}\setminus t\mathbb{Z})\times t\mathbb{Z}$. 
 From Eq.\eqref{eq:4.2}, it can be inferred that $,
\forall~\textbf{m},\textbf{n}\in\Gamma$,
\begin{equation}\label{eq:1}
2\left(q^{m_{2}n_{1}}-q^{m_{1}n_{2}}\right) \alpha^{\bf i}_{\textbf{m}+\textbf{n}}=
\alpha^{\bf i}_{\textbf{m}}\left(q^{m_{2}n_{1}}-q^{(m_{1}+i_{1})n_{2}}\right)+\alpha^{\bf i}_{\textbf{n}}\left(q^{m_{2}(n_{1}+i_{1})}-q^{m_{1}n_{2}}\right).\end{equation}
Particularly, taking $\textbf{n}= \textbf{e}_{1}$  and $\textbf{e}_{2}$ in Eq.\eqref{eq:1}, respectively, we have
\begin{equation} 2\left(1-q^{m_{2}}\right) \alpha_{\textbf{m}+\textbf{e}_{1}}^{\textbf{i}}=
\alpha_{\textbf{m}}^{\textbf{i}}\left(1-q^{m_{2}}\right)+\alpha_{\textbf{e}_{1}}^{\textbf{i}}\left(1-q^{m_{2}(1+i_{1})}\right),~\forall~\textbf{m}\in\Gamma,\label{eq:2}\end{equation}
\begin{equation} 2\left(q^{m_{1}}-1\right) \alpha_{\textbf{m}+\textbf{e}_{2}}^{\textbf{i}}=
\alpha_{\textbf{m}}^{\textbf{i}}\left(q^{m_{1}+i_{1}}-1\right)+\alpha_{\textbf{e}_{2}}^{\textbf{i}}\left(q^{m_{1}}-q^{m_{2}i_{1}}\right),~\forall~\textbf{m}\in\Gamma.\label{eq:3}\end{equation}
Fix $m_{2}\in \mathbb{Z}\setminus t\mathbb{Z}$ and treat $\left (\alpha_{\textbf{m}}^{\textbf{i}}-
\frac{\alpha_{\textbf{e}_{1}}^{\textbf{i}}\left(1-q^{m_{2}(1+i_{1})}\right)}{1-q^{m_{2}}} \right ) _{m_{1}\in \mathbb{Z}}$ as a geometric sequence, then by Eq.\eqref{eq:2}, we get $\forall~\textbf{m}\in \mathbb{Z}\times(\mathbb{Z}\setminus t\mathbb{Z})$,
\begin{equation}\label{eq:4}
\alpha_{\textbf{m}}^{\textbf{i}}=\left(\alpha_{(0, m_{2})}^{\textbf{i}}-
\frac{\alpha_{\textbf{e}_{1}}^{\textbf{i}}\left(1-q^{m_{2}(1+i_{1})}\right)}{1-q^{m_{2}}}\right)\left(\frac{1}{2}\right)^{m_{1}}+
\frac{\alpha_{\textbf{e}_{1}}^{\textbf{i}}\left(1-q^{m_{2}(1+i_{1})}\right)}{1-q^{m_{2}}}.\end{equation}
By taking $m_{1} = 0$ and $m_{1}\in t\mathbb{Z}\setminus\left \{ 0 \right \} $ in Eq.\eqref{eq:3}, respectively, we get
\begin{equation}\label{eq:5}
\alpha_{(0, m_{2})}^{\textbf{i}}=\frac{\alpha_{\textbf{e}_{2}}^{\textbf{i}}\left(1-q^{m_{2}i_{1}}\right)}{1-q^{i_{1}}},~\forall~m_{2}\in \mathbb{Z}^{*},\end{equation}
\begin{equation}\alpha_{\textbf{m}}^{\textbf{i}}=\frac{\alpha_{\textbf{e}_{2}}^{\textbf{i}}\left(1-q^{m_{2} i_{1}}\right)}{1-q^{i_{1}}},~\forall~ \textbf{m}\in (t\mathbb{Z})^* \times\mathbb{Z}.\label{eq:333}\end{equation}
Through the above two equations, we know
\begin{equation}\alpha_{(0, m_{2})}^{\textbf{i}}=\alpha_{(t, m_{2})}^{\textbf{i}}=\alpha_{(2t, m_{2})}^{\textbf{i}}\cdots,~\forall~m_{2}\in \mathbb{Z}^{*}.\label{eq:8.11}\end{equation}
Taking $m_{1}=t$ and $m_{1}=2t$, respectively, by Eq.\eqref{eq:4}, we have $\forall~m_{2}\in \mathbb{Z}\setminus t\mathbb{Z},$
\begin{equation}\label{eq:8.12}
\alpha_{(t, m_{2})}^{\textbf{i}}=\left(\alpha_{(0, m_{2})}^{\textbf{i}}-
\frac{\alpha_{\textbf{e}_{1}}^{\textbf{i}}\left(1-q^{m_{2}(1+i_{1})}\right)}{1-q^{m_{2}}}\right)\left(\frac{1}{2}\right)^{t}+
\frac{\alpha_{\textbf{e}_{1}}^{\textbf{i}}\left(1-q^{m_{2}(1+i_{1})}\right)}{1-q^{m_{2}}},
\end{equation}
\begin{equation}\label{eq:8.13}
\alpha_{(2t, m_{2})}^{\textbf{i}}=\left(\alpha_{(0, m_{2})}^{\textbf{i}}-
\frac{\alpha_{\textbf{e}_{1}}^{\textbf{i}}\left(1-q^{m_{2}(1+i_{1})}\right)}{1-q^{m_{2}}}\right)\left(\frac{1}{2}\right)^{2t}+
\frac{\alpha_{\textbf{e}_{1}}^{\textbf{i}}\left(1-q^{m_{2}(1+i_{1})}\right)}{1-q^{m_{2}}}.
\end{equation}
By Eq.\eqref{eq:8.11}, Eq.\eqref{eq:8.12} and Eq.\eqref{eq:8.13}, we get
\begin{equation}\alpha_{(0, m_{2})}^{\textbf{i}}=\frac{\alpha_{\textbf{e}_{1}}^{\textbf{i}}\left(1-q^{m_{2}(1+i_{1})}\right)}{1-q^{m_{2}}},~\forall~m_{2}\in \mathbb{Z}\setminus t\mathbb{Z}.\label{eq:8.14}\end{equation}
By substituting  Eq.\eqref{eq:8.14} into Eq.\eqref{eq:4}, we have
\begin{equation}\alpha_{\textbf{m}}^{\textbf{i}}=\frac{\alpha_{\textbf{e}_{1}}^{\textbf{i}}\left(1-q^{m_{2}(1+i_{1})}\right)}{1-q^{m_{2}}},~\forall~\textbf{m}\in \mathbb{Z}\times(\mathbb{Z}\setminus t\mathbb{Z}).\label{eq:8.20}\end{equation}
Combining Eq.\eqref{eq:5}, Eq.\eqref{eq:8.14} and Eq.\eqref{eq:8.20}, we have
\begin{equation}\label{eq:8.30}
\alpha_{\textbf{m}}^{\textbf{i}}=\frac{\alpha_{\textbf{e}_{2}}^{\textbf{i}}\left(1-q^{m_{2} i_{1}}\right)}{1-q^{i_{1}}},~\forall~\textbf{m}\in \mathbb{Z}\times(\mathbb{Z}\setminus t\mathbb{Z}).
\end{equation}
By substituting Eq.\eqref{eq:8.30} into Eq.\eqref{eq:3}, we have
for those $\textbf{m}\in \mathbb{Z}\times(\mathbb{Z}\setminus t\mathbb{Z})$ such that ${\bf m}+{\bf e_2}\in \mathbb Z\times (\mathbb Z\setminus t\mathbb{Z})$(since $t\neq 2,$ such  ${\bf  m}$'s exist),
\begin{equation}\label{eq:8.15}
2\left(q^{m_{1}}-1\right) \frac{\alpha_{\textbf{e}_{2}}^{\textbf{i}}\left(1-q^{\left(m_{2}+1\right) i_{1}}\right)}{1-q^{i_{1}}}=
\frac{\alpha_{\textbf{e}_{2}}^{\textbf{i}}\left(1-q^{m_{2} i_{1}}\right)}{1-q^{i_{1}}}\left(q^{m_{1}+i_{1}}-1\right)+
\alpha_{\textbf{e}_{2}}^{\textbf{i}}\left(q^{m_{1}}-q^{m_{2}i_{1}}\right).\end{equation}
Particularly, by taking $m_{1}=i_{1}$ in Eq.\eqref{eq:8.15}, we know
\begin{equation*}\left(q^{(m_{2}+1)i_{1}}-1\right)\alpha_{\textbf{e}_{2}}^{\textbf{i}}=0,~\forall~m_{2}\in \mathbb{Z}\setminus t\mathbb{Z}.\label{eq:8.16}\end{equation*}
Since $q^{i_{1}}\not=1$, for $t\neq2$, there exists an $m_{2}\in \mathbb{Z}\setminus t\mathbb{Z}$ such that $q^{(m_{2}+1)i_{1}}-1\not=0$, we have
\begin{equation*}\alpha_{\textbf{e}_{2}}^{\textbf{i}}=0.\label{eq:8.42}\end{equation*}
By substituting $\alpha_{\textbf{e}_{2}}^{\textbf{i}}=0$ into Eq.\eqref{eq:8.30}, we have
\begin{equation}\label{eq:3.49}
\alpha_{\textbf{m}}^{\textbf{i}}=0,~\forall~ \textbf{m}\in \mathbb{Z}\times(\mathbb{Z}\setminus t\mathbb{Z}).
\end{equation}
And by Eq.\eqref{eq:5} and Eq.\eqref{eq:333}, we have
\begin{equation*}
\alpha_{\textbf{m}}^{\textbf{i}}=0,~\forall~\textbf{m}\in (t\mathbb{Z}\times\mathbb{Z})\setminus\left\{{\bf 0}\right\}.
\end{equation*}
For all ${\bf m}\in (\mathbb Z\times t\mathbb Z)\setminus\{(0,0)\}, {\bf m}-{\bf e_2}\in \mathbb Z\times (\mathbb Z\setminus t\mathbb Z),$ then by substituting $\alpha_{\textbf{e}_{2}}^{\textbf{i}}=0$ and Eq.\eqref{eq:3.49} into Eq.\eqref{eq:3},  we get
$$2\left(q^{m_{1}}-1\right)\alpha_{\bf m}^{\textbf{i}}=0,~\forall~{\bf m}\in(\mathbb Z\times t\mathbb Z)\setminus\{{\bf 0}\}.$$
then
$$\alpha_{\bf m}^{\textbf{i}}=0,~\forall~{\bf m}\in\left(\mathbb{Z}\setminus t\mathbb{Z}\right)\times t\mathbb{Z}.$$

In summary, for $\textbf{i}\in (\mathbb{Z}\setminus t\mathbb{Z})\times t\mathbb{Z}$ or $t\mathbb{Z}\times (\mathbb{Z}\setminus t\mathbb{Z})$, $\forall \textbf{m}\in\mathbb{Z}^{2}\setminus\left\{{\bf 0}\right\}$, $\alpha_{\textbf{m}}^{\textbf{i}}=0$,  and $\varphi_{\textbf{i}} \left(L_{\textbf{m}}\right)=0.$

\textbf{Case 3.} $\textbf{i}=(i_{1},i_{2})\in(\mathbb{Z}\setminus t\mathbb{Z} )\times(\mathbb{Z}\setminus t\mathbb{Z})$. 

By taking $\textbf{n}= \textbf{e}_{1}$  and $\textbf{e}_{2}$ in Eq.\eqref{eq:4.2}, respectively, we have
\begin{equation} \label{eq:0.1}
2\left(1-q^{m_{2}}\right) \alpha_{\textbf{m}+\textbf{e}_{1}}^{\textbf{i}}=
\alpha_{\textbf{m}}^{\textbf{i}}\left(1-q^{m_{2}+i_{2}}\right)+\alpha_{\textbf{e}_{1}}^{\textbf{i}}\left(q^{m_{1} i_{2}}-q^{m_{2}(1+i_{1})}\right),~\forall~\textbf{m}\in\Gamma,
\end{equation}
\begin{equation}\label{eq:0.2}
 2\left(q^{m_{1}}-1\right) \alpha_{\textbf{m}+\textbf{e}_{2}}=
\alpha_{\textbf{m}}^{\textbf{i}}\left(q^{m_{1}+i_{1}}-1\right)+\alpha_{\textbf{e}_{2}}^{\textbf{i}}\left(q^{m_{1}(1+i_{2})}-q^{m_{2} i_{1}}\right),~\forall~\textbf{m}\in\Gamma.\end{equation}
By taking $m_{2}\in t\mathbb{Z}$ in Eq.\eqref{eq:0.1}, we get
\begin{equation}\label{eq:0.51}
\alpha_{\textbf{m}}^{\textbf{i}}=\frac{\alpha_{\textbf{e}_{1}}^{\textbf{i}}\left(1-q^{m_{1} i_{2}}\right)}{1-q^{i_{2}}},~\forall~\textbf{m}\in(\mathbb{Z}\times t\mathbb{Z})\setminus\left\{{\bf 0}\right\}.\end{equation}
Particularly,
$$\alpha_{( m_{1},0)}^{\textbf{i}}=\frac{\alpha_{\textbf{e}_{1}}^{\textbf{i}}\left(1-q^{m_{1} i_{2}}\right)}{1-q^{i_{2}}},~\forall~m_{1}\in\mathbb{Z}^{*}.$$
By taking $m_{1}\in t\mathbb{Z}$ in Eq.\eqref{eq:0.2}, we get
\begin{equation}\label{eq:0.4}
\alpha_{\textbf{m}}^{\textbf{i}}=\frac{\alpha_{\textbf{e}_{2}}^{\textbf{i}}\left(1-q^{m_{2} i_{1}}\right)}{1-q^{i_{1}}},~\forall~\textbf{m}\in( t\mathbb{Z}\times\mathbb{Z} )\setminus\left\{{\bf 0}\right\}.
\end{equation}
Particularly,
$$\alpha_{(0, m_{2})}^{\textbf{i}}=\frac{\alpha_{\textbf{e}_{2}}^{\textbf{i}}\left(1-q^{m_{2} i_{1}}\right)}{1-q^{i_{1}}},~\forall ~m_{2}\in\mathbb{Z}^{*}.$$
By taking  $n_{1} = 0$, $m_{2} = 0$ in Eq.\eqref{eq:4.2}, we get $\forall~(m_{1},n_{2})\in\left(\mathbb{Z}\setminus t\mathbb{Z}\right)\times\left(\mathbb{Z}\setminus t\mathbb{Z}\right)$,
\begin{equation} \label{eq:0.00000}
2\left(1-q^{m_{1}n_{2}}\right) \alpha_{(m_{1}, n_{2})}^{\textbf{i}}=
\alpha_{( m_{1},0)}^{\textbf{i}}\left(1-q^{n_{2}(m_{1}+i_{1})}\right)+\alpha_{(0, n_{2})}^{\textbf{i}}\left(1-q^{m_{1}(n_{2}+i_{2})}\right),\end{equation}
By substituting $\alpha_{( m_{1},0)}^{\textbf{i}}$ and $\alpha_{(0, m_{2})}^{\textbf{i}}$ into the above equation, we have $\forall~(m_{1},n_{2})\in\left(\mathbb{Z}\setminus t\mathbb{Z}\right)\times\left(\mathbb{Z}\setminus t\mathbb{Z}\right)$,
\begin{eqnarray} 
&&2\left(1-q^{m_{1}n_{2}}\right) \alpha_{(m_{1}, n_{2})}^{\textbf{i}}\\
&=&
\frac{\alpha_{\textbf{e}_{1}}^{\textbf{i}}\left(1-q^{m_{1} i_{2}}\right)\left(1-q^{n_{2}(m_{1}+i_{1})}\right)}{1-q^{i_{2}}}+\frac{\alpha_{\textbf{e}_{2}}^{\textbf{i}}\left(1-q^{n_{2} i_{1}}\right)\left(1-q^{m_{1}(n_{2}+i_{2})}\right)}{1-q^{i_{1}}}.\nonumber
\end{eqnarray}
By substituting the above equation into Eq.\eqref{eq:4.2}, for those $\textbf{m},\textbf{n}\in\left(\mathbb{Z}\setminus t\mathbb{Z}\right)\times\left(\mathbb{Z}\setminus t\mathbb{Z}\right)$ such that $\textbf{m}+\textbf{n}\in\left(\mathbb{Z}\setminus t\mathbb{Z}\right)\times\left(\mathbb{Z}\setminus t\mathbb{Z}\right)$, we have
\begin{eqnarray*}
&&2\lambda\left(\textbf{m},\textbf{n}\right)
(\frac{\left(1-q^{(m_{2}+n_{2})\left(m_{1}+n_{1}+i_{1}\right)}\right)\left(1-q^{(m_{1}+n_{1}) i_{2}}\right)}{(1-q^{i_{2}})\left(1-q^{(m_{1}+n_{1})(m_{2}+n_{2})}\right)}\alpha_{\textbf{e}_{1}}\\
&&+\frac{\left(1-q^{(m_{1}+n_{1}) \left(m_{2}+n_{2}+i_{2} \right)}\right)\left(1-q^{(m_{2}+n_{2}) i_{1}}\right)}{(1-q^{i_{1})}\left(1-q^{(m_{1}+n_{1})(m_{2}+n_{2})}\right)}\alpha_{\textbf{e}_{2}})\\
&=&\lambda\left(\textbf{m}+\textbf{i},\textbf{n}\right)
\left(\frac{\left(1-q^{m_{2}\left(m_{1}+i_{1}\right)}\right)\left(1-q^{m_{1}i_{2}}\right)}{(1-q^{i_{2}})\left(1-q^{m_{1}m_{2}}\right)}\alpha_{\textbf{e}_{1}}+\frac{\left(1-q^{m_{1}\left(m_{2}+i_{2} \right)}\right)\left(1-q^{m_{2} i_{1}}\right)}{(1-q^{i_{1})}\left(1-q^{m_{1}m_{2}}\right)}\alpha_{\textbf{e}_{2}}\right)\\
&&+\lambda\left(\textbf{m},\textbf{n}+\textbf{i}\right)
\left(\frac{\left(1-q^{n_{2}\left(n_{1}+i_{1}\right)}\right)\left(1-q^{n_{1}i_{2}}\right)}{(1-q^{i_{2}})\left(1-q^{n_{1}n_{2}}\right)}\alpha_{\textbf{e}_{1}}+\frac{\left(1-q^{n_{1}\left(n_{2}+i_{2} \right)}\right)\left(1-q^{n_{2} i_{1}}\right)}{(1-q^{i_{1})}\left(1-q^{n_{1}n_{2}}\right)}\alpha_{\textbf{e}_{2}}\right).
\end{eqnarray*}
By the same argument as in Case 3 of Theorem\ref{E}, we know that
$$\alpha_{\textbf{e}_{1}}=\alpha_{\textbf{e}_{2}}=0.$$
And we know immediately that $\forall m_1, m_2\in\mathbb Z^*, \alpha_{(m_1,0)}=\alpha_{(0,m_2)}=0.$
By substituting $\alpha_{(m_1,0))}=0$ and $\alpha_{(0,m_2))}=0$ into Eq.\eqref{eq:0.00000}, we get
$$\alpha_{(m_{1},n_{2})}=0,~\forall ~(m_{1},n_{2})\in\left(\mathbb{Z}\setminus t\mathbb{Z}\right)\times\left(\mathbb{Z}\setminus t\mathbb{Z}\right).$$
i.e.
\begin{equation*}
\alpha_{\textbf{m}}=0,~\forall~ \textbf{m}\in\left(\mathbb{Z}\setminus t\mathbb{Z}\right)\times\left(\mathbb{Z}\setminus t\mathbb{Z}\right).\end{equation*}
By substituting $\alpha_{\textbf{e}_{1}}=0$ into Eq.\eqref{eq:0.51}, we get
\begin{equation*}\alpha_{\textbf{m}}=0,~\forall~\textbf{m}\in(\mathbb{Z}\times t\mathbb{Z})\setminus\left\{{\bf 0}\right\}.\end{equation*}
By substituting $\alpha_{\textbf{e}_{2}}=0$ into Eq.\eqref{eq:0.4}, we get
\begin{equation*}\alpha_{\textbf{m}}=0,~\forall~\textbf{m}\in( t\mathbb{Z}\times\mathbb{Z} )\setminus\left\{{\bf 0}\right\}.\end{equation*}
So
$$\alpha_{\textbf{m}}^{\textbf{i}}=0,~\forall~\textbf{m}\in\mathbb{Z}^{2}\setminus\left\{{\bf 0}\right\}.$$

In summary, for $\textbf{i}\in(\mathbb{Z}\setminus t\mathbb{Z} )\times(\mathbb{Z}\setminus t\mathbb{Z})$,   $\forall \textbf{m}\in\mathbb{Z}^{2}\setminus\left\{{\bf 0}\right\}$, $\alpha_{\textbf{m}}^{\textbf{i}}=0$, and $\varphi_{\textbf{i}} \left(L_{\textbf{m}}\right)=0.$

Hence combining the analysis of the three cases above, we obtain the desired result.
\end{proof}

Based on Lemma \ref{C} and Theorem \ref{F}. We can provide a comprehensive characterization of transposed Poisson algebra structures on the  algebra $\tilde{V} _{q}$.

\begin{theorem} \label{G}
Let $\left( \tilde{V} _{q},\cdot,\left[\cdot ,\cdot \right]\right)$ be a transposed Poisson structure defined on $\tilde{V} _{q}$. Then the multiplication on $\left( \tilde{V} _{q},\cdot\right)$ is given
by:
$$L _{\bf m}\cdot L _{\bf n}=\left\{\begin{matrix}
  0,   &  {\bf m}-{\bf n}\in\Gamma_{2}, \\
  \underset{{\bf i}\in \Gamma_{1}}{\Sigma}\alpha^{\bf i}_{\bf m}L_{{\bf n}+{\bf i}}, & {\bf m},{\bf n}\in\Gamma_{2}
  ~~\mathrm{such~that}~~{\bf m}-{\bf n}\in\Gamma_{1},\\
\underset{{\bf i}\in \Gamma_{1}}{\Sigma}\alpha^{\bf i}_{{\bf m},{\bf n}}L_{{\bf n}+{\bf i}},&{\bf m},{\bf n}\in\Gamma_{1}.
\end{matrix}\right.$$
where
$$\alpha^{\bf i}_{\bf m}=\alpha^{({\bf n}-{\bf m})+{\bf i}}_{\bf n},~\forall~{\bf m},{\bf n}\in\Gamma_{2}~
  such~ that~ {\bf m}-{\bf n}\in\Gamma_{1},~\forall ~{\bf i}\in \Gamma_{1},$$
$$\alpha^{\bf i}_{{\bf m},{\bf n}}=\alpha^{\left({\bf n}-{\bf m}\right)+{\bf i}}_{{\bf n},{\bf m}},~\forall~{\bf m},{\bf n}\in\Gamma_{2},~\forall~ {\bf i}\in \Gamma_{1}\setminus \left \{ -{\bf n} \right \} ,$$
and when at least one of ${\bf m}, {\bf n}, {\bf r}\in \Gamma_{1}$ is not equal to the others,
$$\underset{{\bf i}\in \Gamma_{1}}{\Sigma}
\alpha^{\bf i}_{{\bf m}, {\bf n}}\alpha^{\bf j}_{{\bf n}+{\bf i}, {\bf r}}=
\underset{{\bf i}\in \Gamma_{1}}{\Sigma}
\alpha^{\bf i}_{{\bf n}, {\bf r}}\alpha^{\bf j}_{{\bf r}+{\bf i}, {\bf m}}=0,~\forall~{\bf j}\in \Gamma_{1}\setminus \left \{ -{\bf m},-{\bf r} \right \}.$$
\end{theorem}

\begin{proof}
Let $\left( \tilde{V} _{q},\cdot,\left[\cdot ,\cdot \right]\right)$ be a transposed Poisson algebra, then $\left( \tilde{V} _{q},\cdot\right)$ is a commutative and associative algebra which satisfies Eq.\eqref{eq: 2.1}. For all $\textbf{m}\in\Gamma$, we denote the left multiplication by $L_{\textbf{m}}$ in $\left( \tilde{V} _{q},\cdot\right)$ as $\varphi ^{\textbf{m}}$, it follows that $\forall~\textbf{n}\in\Gamma$, $L _{\textbf{m}}\cdot L _{\textbf{n}}=\varphi ^{\textbf{m}}\left(L _{\textbf{n}}\right)$. Since $\cdot$ is commutative, we also have
$\forall~\textbf{m},\textbf{n}\in\Gamma$, $\varphi ^{\textbf{m}}\left(L _{\textbf{n}}\right)=\varphi ^{\textbf{n}}\left(L _{\textbf{m}}\right)$. By Lemma \ref{Q}, we know $\forall\textbf{m},\textbf{n}\in\Gamma $, $\varphi ^{\textbf{m}}$, $\varphi ^{\textbf{n}}\in\bigtriangleup\left(\tilde{V} _{q}\right)$. According to Theorem \ref{F}, for all $\textbf{m}\in\Gamma$, we write $$\varphi^{\textbf{m}} \left(L_{\textbf{n}}\right)=\left\{\begin{matrix}
 \underset{\textbf{i}\in \Gamma_{1}}{\Sigma}\alpha^{\textbf{i}}_{\textbf{m}} L_{\textbf{n}+\textbf{i}}, & \textbf{n}\in\Gamma_{2},\\
  \underset{\textbf{i}\in \Gamma_{1}}{\Sigma}\alpha^{\textbf{i}}_{\textbf{m},\textbf{n}} L_{\textbf{n}+\textbf{i}},&\textbf{n}\in\Gamma_{1}
\end{matrix}\right.$$
and $$\varphi ^{\textbf{n}}\left(L _{\textbf{m}}\right)=\left\{\begin{matrix}
 \underset{\textbf{i}\in \Gamma_{1}}{\Sigma}\alpha^{\textbf{i}}_{\textbf{n}} L_{\textbf{m}+\textbf{i}}, & \textbf{m}\in\Gamma_{2},\\
  \underset{\textbf{i}\in \Gamma_{1}}{\Sigma}\alpha^{\textbf{i}}_{\textbf{n},\textbf{m}} L_{\textbf{m}+\textbf{i}},&\textbf{m}\in\Gamma_{1}.
\end{matrix}\right.$$

Now we need to consider the following cases:

\textbf{Case 1.} $\textbf{m}-\textbf{n}\in \Gamma_{2}$.

\textbf{Subcase 1.} $\textbf{m},\textbf{n}\in \Gamma_{2}$. 

We know that
$$0=L _{\textbf{m}}\cdot L _{\textbf{n}}-L _{\textbf{n}}\cdot L _{\textbf{m}}=\varphi ^{\textbf{m}}\left(L _{\textbf{n}}\right)-\varphi ^{\textbf{n}}\left(L _{\textbf{m}}\right)=\underset{\textbf{i}\in \Gamma_{1}}{\Sigma}
\alpha^{\textbf{i}}_{\textbf{m}}L_{\textbf{n}+\textbf{i}}-\underset{\textbf{i}\in \Gamma_{1}}{\Sigma}
\alpha^{\textbf{i}}_{\textbf{n}}L_{\textbf{m}+\textbf{i}}.$$
Since $\textbf{m}-\textbf{n}\in \Gamma_{2}$, then $\textbf{m}\not=\textbf{n}$ and for all $\textbf{i}\in \Gamma_{1}$, $\textbf{n}+\textbf{i}\not=\textbf{m}+\textbf{i}$ and $\left \{ L_{\textbf{n}+\textbf{i}},L_{\textbf{m}+\textbf{i}} \right \}_{\textbf{i}\in\Gamma_{1}} $ is linear independent, it follows that $\forall~\textbf{i}\in\Gamma_{1}$,
$\alpha^{\textbf{i}}_{\textbf{m}}=\alpha^{\textbf{i}}_{\textbf{n}}=0$ and $L _{\textbf{m}}\cdot L _{\textbf{n}}=L _{\textbf{n}}\cdot L _{\textbf{m}}=0$.

\textbf{Subcase 2.} $\textbf{m}\in \Gamma_{1}$, $\textbf{n}\in \Gamma_{2}$ or $\textbf{m}\in \Gamma_{2}$, $\textbf{n}\in \Gamma_{1}$. 

Without loss of generality, we suppose $\textbf{m}\in \Gamma_{1}$, $\textbf{n}\in \Gamma_{2}$. Then $\textbf{m}\not=\textbf{n}$, and for all $\textbf{i}\in \Gamma_{1}$, $\textbf{n}+\textbf{i}\not=\textbf{m}+\textbf{i}$. Since $$0=L _{\textbf{m}}\cdot L _{\textbf{n}}-L _{\textbf{n}}\cdot L _{\textbf{m}}=\varphi ^{\textbf{m}}\left(L _{\textbf{n}}\right)-\varphi ^{\textbf{n}}\left(L _{\textbf{m}}\right)=\underset{\textbf{i}\in \Gamma_{1}}{\Sigma}
\alpha^{\textbf{i}}_{\textbf{m}}L_{\textbf{n}+\textbf{i}}-\underset{\textbf{i}\in \Gamma_{1}}{\Sigma}\alpha^{\textbf{i}}_{\textbf{n},\textbf{m}}L_{\textbf{m}+\textbf{i}},$$ by the same argument as in Subcase 1, we get \begin{center}$\forall~\textbf{i}\in\Gamma_{1}$,
$\alpha^{\textbf{i}}_{\textbf{m}}=\alpha^{\textbf{i}}_{\textbf{n},\textbf{m}}=0$ and $L _{\textbf{m}}\cdot L _{\textbf{n}}=L _{\textbf{n}}\cdot L _{\textbf{m}}=0$.\end{center}

\textbf{Case 2.} $\textbf{m}-\textbf{n}\in \Gamma_{1}$.

\textbf {Subcase 1.} $\textbf{m},\textbf{n}\in \Gamma_{2}$.
 
 We know that
\begin{eqnarray*}
0&=&L _{\textbf{m}}\cdot L _{\textbf{n}}-L _{\textbf{n}}\cdot L _{\textbf{m}}
=\underset{\textbf{i}\in \Gamma_{1}}{\Sigma}
\alpha^{\textbf{i}}_{\textbf{m}}L_{\textbf{n}+\textbf{i}}-\underset{\textbf{j}\in \Gamma_{1}}{\Sigma}
\alpha^{\textbf{j}}_{\textbf{n}}L_{\textbf{m}+\textbf{j}}\\
&=&\underset{\textbf{i}\in \Gamma_{1}}{\Sigma}\alpha^{\textbf{i}}_{\textbf{m}}L_{\textbf{n}+\textbf{i}}
-\underset{\textbf{i}\in \Gamma_{1}}{\Sigma}\alpha^{\textbf{i+(n-m)}}_{\textbf{n}}L_{\textbf{n}+\textbf{i}}
=\underset{\textbf{i}\in \Gamma_{1}}{\Sigma}\left(\alpha^{\textbf{i}}_{\textbf{m}}-\alpha^{\textbf{i+(n-m)}}_{\textbf{n}}\right)L_{\textbf{n}+\textbf{i}}.\end{eqnarray*}
Since $\left \{ L_{\textbf{n}+\textbf{i}} \right \}_{\textbf{i}\in\Gamma_{1}} $ is linear independent, then $\forall~\textbf{i}\in \Gamma_{1}$, $\alpha^{\textbf{i}}_{\textbf{m}}=\alpha^{\textbf{i+(n-m)}}_{\textbf{n}}$.

\textbf{Subcase 2.} $\textbf{m},\textbf{n}\in \Gamma_{1}$. 

We know that
\begin{eqnarray*}
0&=&L _{\textbf{m}}\cdot L _{\textbf{n}}-L _{\textbf{n}}\cdot L _{\textbf{m}}
=\underset{\textbf{i}\in \Gamma_{1}}{\Sigma}
\alpha^{\textbf{i}}_{\textbf{m},\textbf{n}}L_{\textbf{n}+\textbf{i}}-\underset{\textbf{j}\in \Gamma_{1}}{\Sigma}
\alpha^{\textbf{j}}_{\textbf{n},\textbf{m}}L_{\textbf{m}+\textbf{j}}\\
&=&\underset{\textbf{i}\in \Gamma}{\Sigma}\alpha^{\textbf{i}}_{\textbf{m},\textbf{n}}L_{\textbf{n}+\textbf{i}}
-\underset{\textbf{i}\in \Gamma_{1}}{\Sigma}\alpha^{\textbf{i+(n-m)}}_{\textbf{n},\textbf{m}}L_{\textbf{n}+\textbf{i}}
=\underset{\textbf{i}\in \Gamma_{1}}{\Sigma}\left(\alpha^{\textbf{i}}_{\textbf{m},\textbf{n}}-\alpha^{\textbf{i+(n-m)}}_{\textbf{n},\textbf{m}}\right)L_{\textbf{n}+\textbf{i}}.\end{eqnarray*}
Since $\left \{ L_{\textbf{n}+\textbf{i}} \right \}_{\textbf{i}\in\Gamma_{1}\setminus \left \{ -\textbf{n} \right \}} $ is linear independent, then $\forall~\textbf{i}\in \Gamma_{1}\setminus \left \{ -\textbf{n} \right \} $, $\alpha^{\textbf{i}}_{\textbf{m},\textbf{n}}=\alpha^{\textbf{i+(n-m)}}_{\textbf{n},\textbf{m}}$.

Next, we consider the associative law. For all $\textbf{m},\textbf{n},\textbf{r}\in\mathbb{Z}^{2}\setminus\left\{\left(0,0\right)\right\}$, we have
\begin{equation}\label{eq:666}
\left(L_{\textbf{m}}\cdot L_{\textbf{n}}\right)\cdot L_{\textbf{r}}=L_{\textbf{m}}\cdot\left(L_{\textbf{n}}\cdot L_{\textbf{r}}\right).
\end{equation}

Based on the preceding analysis, it is easy to see if one or two of $\textbf{m}, \textbf{n}$ and $\textbf{r}$ belong to $\Gamma_2,$ then both sides of Eq.\eqref{eq:666} are zero. So, it suffices to consider the scenarios where $\textbf{m},\textbf{n},\textbf{r}\in\Gamma_{2}$ and $\textbf{m},\textbf{n},\textbf{r}\in\Gamma_{1}$. 
Now we consider these two cases.

\textbf{Case 1.} $\textbf{m},\textbf{n},\textbf{r}\in\Gamma_{2}$.

If $\textbf{m}-\textbf{n}\in \Gamma_2$ or $\textbf{n}-\textbf{r}\in \Gamma_2,$ then both sides of Eq.\eqref{eq:666} are zero. Now we suppose $\textbf{m}-\textbf{n}, \textbf{n}-\textbf{r}\in \Gamma_1.$ According to Theorem \ref{F}, we have 
\begin{eqnarray*}
(L_{\bf m}\cdot L_{\bf n})\cdot L_{\bf r}&=&(\sum\limits_{{\bf i}\in \Gamma_1}\alpha_{\bf m}^{\bf i}L_{\bf n+i})\cdot L_{\bf r}
=\sum\limits_{{\bf i}\in \Gamma_1}\alpha_{\bf m}^{\bf i}L_{\bf r}\cdot L_{\bf n+i}\\
&=&\sum\limits_{{\bf i}\in \Gamma_1}\alpha_{\bf m}^{\bf i}\sum\limits_{{\bf j}\in \Gamma_1}\alpha_{\bf r}^{\bf j}L_{\bf n+i+j}=\sum\limits_{{\bf i, j}\in \Gamma_1}\alpha_{\bf m}^{\bf i}\alpha_{\bf r}^{\bf j}L_{\bf n+i+j}
\end{eqnarray*}
and 
\begin{eqnarray*}
L_{\bf m}\cdot (L_{\bf n}\cdot L_{\bf r})&=&L_{\bf m}\cdot (L_{\bf r}\cdot L_{\bf n})
=L_{\bf m}\cdot\sum\limits_{{\bf j}\in \Gamma_1}\alpha_{\bf r}^{\bf j}L_{\bf n+j}\\
&=&\sum\limits_{{\bf j}\in \Gamma_1}\alpha_{\bf r}^{\bf j}  \sum\limits_{{\bf i}\in \Gamma_1}\alpha_{\bf m}^{\bf i}L_{\bf n+i+j}=\sum\limits_{{\bf i, j}\in \Gamma_1}\alpha_{\bf r}^{\bf j}\alpha_{\bf m}^{\bf i}L_{\bf n+i+j}.
\end{eqnarray*}
Thus, $(L_{\bf m}\cdot L_{\bf n})\cdot L_{\bf r}=L_{\bf m}\cdot (L_{\bf n}\cdot L_{\bf r}).$

\textbf{Case 2.} $\textbf{m},\textbf{n},\textbf{r}\in\Gamma_{1}$.

\textbf{Subcase 1.} $\textbf{m}=\textbf{n}=\textbf{r}.$

It is trivial.

\textbf{Subcase 2.} Two of $\textbf{m},\textbf{n},\textbf{r}$ are equal.
 
 Without loss of generality, we suppose $\textbf{n}=\textbf{r}$ and $\textbf{m}\not=\textbf{n}.$ By the associativity and commutativity, we have
$$\begin{aligned}
0&=\left(L_{\textbf{m}}\cdot L_{\textbf{n}}\right)\cdot L_{\textbf{n}}-L_{\textbf{m}}\cdot\left(L_{\textbf{n}}\cdot L_{\textbf{n}}\right)\\
&= \left(L_{\textbf{m}}\cdot L_{\textbf{n}}\right)\cdot L_{\textbf{n}}-\left(L_{\textbf{n}}\cdot L_{\textbf{n}}\right)\cdot L_{\textbf{m}}\\
&=\left ( \underset{\textbf{i}\in \Gamma_{1}}{\Sigma}\alpha^{\textbf{i}}_{\textbf{m},\textbf{n}}L_{\textbf{n}+\textbf{i}} \right )\cdot L_{\textbf{n}}
-\left ( \underset{\textbf{i}\in \Gamma_{1}}{\Sigma}\alpha^{\textbf{i}}_{\textbf{n},\textbf{n}}L_{\textbf{n}+\textbf{i}} \right ) \cdot L_{\textbf{m}}\\
&=\underset{\textbf{j}\in \Gamma_{1} }{\Sigma}\underset{\textbf{i}\in \Gamma_{1}}{\Sigma}\alpha^{\textbf{i}}_{\textbf{m},\textbf{n}}\alpha^{\textbf{j}}_
{\textbf{n}+\textbf{i},\textbf{n}}L_{\textbf{n}+\textbf{j}}
-\underset{\textbf{j}\in \Gamma_{1} }{\Sigma}\underset{\textbf{i}\in \Gamma_{1}}{\Sigma}\alpha^{\textbf{i}}_{\textbf{n},\textbf{n}}\alpha^{\textbf{j}}_
{\textbf{n}+\textbf{i},\textbf{m}}L_{\textbf{m}+\textbf{j}}
 \\
&=\underset{\textbf{j}\in \Gamma_{1} }{\Sigma}\left (\left ( \underset{\textbf{i}\in \Gamma_{1}}{\Sigma}\alpha^{\textbf{i}}_{\textbf{m},\textbf{n}}\alpha^{\textbf{j}}_
{\textbf{n}+\textbf{i},\textbf{n}}\right)L_{\textbf{n}+\textbf{j}}
-\left (\underset{\textbf{i}\in \Gamma_{1}}{\Sigma}\alpha^{\textbf{i}}_{\textbf{n},\textbf{n}}\alpha^{\textbf{j}}_{\textbf{n}+\textbf{i},\textbf{m}}\right)L_{\textbf{m}+\textbf{j}}\right).
\end{aligned}$$
Since $\textbf{m}\not=\textbf{n}$, then $\forall~\textbf{j}\in \Gamma_{1}$, $\textbf{m}+\textbf{j}\not=\textbf{n}+\textbf{j}$ and $\left \{ L_{\textbf{n}+\textbf{j}},L_{\textbf{m}+\textbf{j}} \right \}_{\textbf{j}\in\Gamma_{1}\setminus \left \{ -\textbf{m},-\textbf{n} \right \}} $ is linear independent. Thus we obtain $\forall~\textbf{j}\in \Gamma_{1}\setminus \left \{ -\textbf{m},-\textbf{n} \right \}$,  $\underset{\textbf{i}\in \Gamma_{1}}{\Sigma}\alpha^{\textbf{i}}_{\textbf{m},\textbf{n}}\alpha^{\textbf{j}}_{\textbf{n}+\textbf{i},\textbf{n}}=\underset{\textbf{i}\in \Gamma_{1}}{\Sigma}\alpha^{\textbf{i}}_{\textbf{n},\textbf{n}}\alpha^{\textbf{j}}_{\textbf{n}+\textbf{i},\textbf{m}}=0$.

\textbf{Case 3.} Among $\textbf{m},\textbf{n},\textbf{r}$, there are no equal pairs.
 
 We know that
$$\begin{aligned}
0&=\left(L_{\textbf{m}}\cdot L_{\textbf{n}}\right)\cdot L_{\textbf{r}}-L_{\textbf{m}}\cdot\left(L_{\textbf{n}}\cdot L_{\textbf{r}}\right)\\
&=\left ( \underset{\textbf{i}\in \Gamma_{1}}{\Sigma}\alpha^{\textbf{i}}_{\textbf{m},\textbf{n}}L_{\textbf{n}+\textbf{i}} \right )\cdot L_{\textbf{r}}
-\left ( \underset{\textbf{i}\in \Gamma_{1}}{\Sigma}\alpha^{\textbf{i}}_{\textbf{n},\textbf{r}}L_{\textbf{r}+\textbf{i}} \right )\cdot L_{\textbf{m}}
\\
&=\underset{\textbf{i}\in \Gamma_{1} }{\Sigma}\alpha^{\textbf{i}}_{\textbf{m},\textbf{n}}\underset{\textbf{j}\in \Gamma_{1}}{\Sigma}\alpha^{\textbf{j}}_
{\textbf{n}+\textbf{i},\textbf{r}}L_{\textbf{r}+\textbf{j}}
-\underset{\textbf{i}\in \Gamma_{1} }{\Sigma}\alpha^{\textbf{i}}_{\textbf{n},\textbf{r}}\underset{\textbf{j}\in \Gamma_{1}}{\Sigma}\alpha^{\textbf{j}}_
{\textbf{r}+\textbf{i},\textbf{m}}L_{\textbf{m}+\textbf{j}}
\\
&=\underset{\textbf{j}\in \Gamma_{1}}{\Sigma}\left ( \left ( \underset{\textbf{i}\in \Gamma_{1} }{\Sigma}\alpha^{\textbf{i}}_{\textbf{m},\textbf{n}}\alpha^{\textbf{j}}_
{\textbf{n}+\textbf{i},\textbf{r}} \right ) L_{\textbf{r}+\textbf{j}}
-\left ( \underset{\textbf{i}\in \Gamma_{1} }{\Sigma}\alpha^{\textbf{i}}_{\textbf{n},\textbf{r}}\alpha^{\textbf{j}}_
{\textbf{r}+\textbf{i},\textbf{m}} \right ) L_{\textbf{m}+\textbf{j}} \right )
\end{aligned}.$$
Since $\textbf{m}\not=\textbf{r}$, then $\forall~\textbf{j}\in \Gamma_{1}$, $\textbf{r}+\textbf{j}\not=\textbf{m}+\textbf{j}$, and $\left \{ L_{\textbf{r}+\textbf{j}},L_{\textbf{m}+\textbf{j}} \right \}_{\textbf{j}\in\Gamma_{1}\setminus \left \{ -\textbf{m},-\textbf{r} \right \}} $ is linear independent. Thus we obtain $\forall~\textbf{j}\in \Gamma_{1}\setminus \left \{ -\textbf{m},-\textbf{r} \right \}$,
$ \underset{\textbf{i}\in \Gamma_{1} }{\Sigma}\alpha^{\textbf{i}}_{\textbf{m},\textbf{n}}\alpha^{\textbf{j}}_{\textbf{n}+\textbf{i},\textbf{r}}
=\underset{\textbf{i}\in \Gamma_{1} }{\Sigma}\alpha^{\textbf{i}}_{\textbf{n},\textbf{r}}\alpha^{\textbf{j}}_{\textbf{r}+\textbf{i},\textbf{m}} =0$.

Based on the analysis above, we obtain the desired result.
\end{proof}

\section{Transposed Poisson structures on $q$-Quantum Torus Lie algebra}\label{sec4}

In this section, we consider the similar problems for the $q$-quantum torus Lie algebra. We first recall the definition for this algebra.

Let $p\in\mathbb{N}^{*}$, a quantum torus $\mathbb{C}_{q} =\mathbb{C}_{q}\left[x_{1}^{\pm 1},\cdots x_{p}^{\pm 1}\right] $ is an associative and noncommutative polynomial algebra subject to the defining relations $x_{i}x_{j}=q_{i,j}x_{j}x_{i}$ and $x_{i}^{-1} x_{i}=x_{i}x_{i}^{-1}=1 $ for $1\le i,j\le p$, where the quantum torus matrix $q=\left(q_{i,j}\right) $ is a $p\times p$ matrix with nonzero complex entries satisfying $q_{i,j}=1$ and $q_{i,j}^{-1}=q_{j,i}$ for $1\le i,j\le p$. 

For the case $p=2$, we identify the quantum torus matrix $q=\left(q_{i,j}\right)_{2\times 2}$ with its entry $q_{1,2}$.
We denote ${\rm Der}_{skew}\left(\mathbb{C}_{q}\right)$ as the skew derivation algebra of $\mathbb{C}_{q}$ and set $ \hat{L(q)}=\mathbb{C}_{q}\oplus {\rm Der}_{skew}\left(\mathbb{C}_{q}\right)$, then $\hat{L( q  )}$ forms a Lie algebra of differential operators acting on the quantum torus $\mathbb{C}_{q}$. Let $L(q)$ be the derived subalgebra of $\hat{L( q  )}$, then $L(q)$ is a perfect Lie algebra.

 
Let $q$ be a $t$th primitive root of unity, then $L(q)$ is isomorphic to a Lie algebra with a basis $\{x^{\bf m}, D(\bf m)\mid m\in \Gamma\}$, where $x^{\textbf{m}}=x_{1}^{m_{1}}x_{2}^{m_{2}} $, 
$D\left(\textbf{m}\right)={\rm ad}x^{\textbf{m}}$,
 for $\textbf{m}\in\Gamma_2$ 
and $D\left(\textbf{m}\right)=x^{\textbf{m}}\left(m_{2}d_{1}-m_{1}d_{2}\right)$ for $\textbf{m}\in\Gamma_{1} $, where $d_{i}=x_{i}\partial_{i} $, $\partial_{i} $ is the usual derivative with respect to variable $x_{i}$, that is $d_{i}(x^{\textbf{m}})=m_{i} x^{\textbf{m}}$.
 And the Lie algebra structure of $L(q)$ is given by the  following anticommutative product: $\forall {\bf m, n}\in \Gamma,$ 
\begin{equation}\left[ x^{\textbf{m}},x^{\textbf{n}} \right ] =\lambda\left(\textbf{m},\textbf{n}\right)x^{\textbf{m+n}},\label{eq:11.0} \end{equation}
\begin{equation}\left[D\left(\textbf{m}\right),x^{\textbf{n}}\right]=h\left(\textbf{m},\textbf{n}\right)x^{\textbf{m+n}},\label{eq:11.1}\end{equation}
\begin{equation}\left[D\left(\textbf{m}\right),D\left(\textbf{n}\right)\right]=g\left(\textbf{m},\textbf{n}\right)D\left(\textbf{m+n}\right),\label{eq:11.2}\end{equation}
where
$$h\left(\textbf{m},\textbf{n}\right)=\left\{\begin{matrix}
 \mathrm{det} \binom{\textbf{n}}{\textbf{m}}, & \textbf{m}\in \Gamma _{1}, \\
 \lambda\left(\textbf{m},\textbf{n}\right), &\textbf{m}\in \Gamma _{2},
\end{matrix}\right.$$

$$g\left(\textbf{m},\textbf{n}\right)=\left\{\begin{matrix}
 \lambda\left(\textbf{m},\textbf{n}\right), & \textbf{m},\textbf{n}\in \Gamma _{2}, \\
 \mathrm{det} \binom{\textbf{n}}{\textbf{m}}, & otherwise
\end{matrix}\right.$$
and $\mathrm{det} \binom{\textbf{n}}{\textbf{m}}=m_{2}n_{1}-m_{1} n_{2}$, $\lambda({\bf m, n})$ is defined in Section \ref{sec3}.
 
When $q$ is generic, that is, $q\in \mathbb C^*$ is not a root of unity, then $\Gamma_1=\emptyset$ and $L(q)$ becomes de the derived Lie subalgebra studied in \cite{bib27}, we denote this algebra as $\tilde{L(q)}$ in this paper. The Lie algebra structure of $\tilde{L(q)}$
 is given by the following anticommutative product:
 \begin{equation}\label{eq:401}
\left[ x^{\textbf{m}},x^{\textbf{n}} \right ] =\lambda\left(\textbf{m},\textbf{n}\right)x^{\textbf{m+n}}, 
\end{equation}
\begin{equation}\label{eq:402}
\left[D\left(\textbf{m}\right),x^{\textbf{n}}\right]=\lambda\left(\textbf{m},\textbf{n}\right)x^{\textbf{m+n}},
\end{equation}
\begin{equation}\label{eq:403}
[D({\bf m}), D({\bf n})]=\lambda({\bf m, n})D({\bf m+n}),
\end{equation}
where $\lambda({\bf m, n})$ is defined in Section \ref{sec3}.

For convenience, in both of the two cases,
we designate $x^{\bf 0}=0
$, $D({\bf 0})=0 $ .
\subsection{The case where $q$ is generic for the $q$-Quantum Torus Lie algebra}\label{subsec4.1}
Observe that $\tilde{L(q)}= \underset{\textbf{m}\in \mathbb{Z}^{2} }{\oplus} \left (\tilde{ L(q)} \right ) _{\textbf{m}}$ is a $\mathbb{Z}^{2}$-grading, where
$\left ( \tilde{L(q)} \right ) _{\bf{m}}=
\mathbb{C} x^{\textbf{m}}+\mathbb CD({\bf m}),$  for $\textbf{m}\in\Gamma.$ In addition, $\tilde{L(q)}$ is finitely generated. 
\begin{lemma}\label{lem4.1}
The $q$-quantum torus Lie algebra $\tilde{L(q)}$ is the Lie algebra generated by the set $\left \{ x^{(0,\pm1)},x^{(\pm 1,0)},D(0,\pm1),D(\pm 1,0) \right \}$, subject to the following relations:
\begin{eqnarray*}
\left[x^{(1,0)},x^{(-1,0)}\right]&=&\left[x^{(0,1)},x^{(0,-1)}\right]=0,\\
\left[D(1,0),D(-1,0)\right]&=&\left[D(0,1),D(0,-1)\right]=0,\\
\left[\left[x^{(1,0)},x^{(0,1)}\right],x^{(-1,0)}\right]&=&(1-q)(q^{-1}-1)x^{(0,1)},\\
\left[\left[x^{(1,0)},x^{(0,-1)}\right],x^{(-1,0)}\right]&=&(1-q^{-1})(q-1)x^{(0,-1)},\\
\left[\left[x^{(0,1)},x^{(1,0)}\right],x^{(0,-1)}\right]&=&(q-1)(1-q^{-1})x^{(1,0)},\\
\left[\left[x^{(0,1)},x^{(-1,0)}\right],x^{(0,-1)}\right]&=&(q^{-1}-1)(1-q)x^{(-1,0)},\\
\left[\left[D(1,0),D(0,1)\right],D(-1,0)\right]&=&(1-q)(q^{-1}-1)D(0,1),\\
\left[\left[D(1,0),D(0,-1)\right],D(-1,0)\right]&=&(1-q^{-1})(q-1)D(0,-1),\\
\left[\left[D(0,1),D(1,0)\right],D(0,-1)\right]&=&(q-1)(1-q^{-1})D(1,0),\\
\left[\left[D(0,1),D(-1,0)\right],D(0,-1)\right]&=&(q^{-1}-1)(1-q)D(-1,0).\end{eqnarray*}
\end{lemma}
\begin{theorem}
Let $\varphi$ be a $\frac{1}{2}$-derivation of $\tilde{L(q)}$. Then there exist $c, d\in \mathbb C$, such that
$$\varphi \left(x^{\textbf{m}}\right)= ( c+d)x^{\textbf{m}},$$
$$\varphi \left(D\left(\textbf{m}\right)\right)=cx^{\textbf{m}}+ dD\left(\textbf{m}\right).$$
\end{theorem}
\begin{proof}Let $\varphi$ be a $\frac{1}{2}$-derivation of $\tilde{L(q)}$, then by Lemma \ref{lem4.1} and Lemma \ref{B}, we can write $\varphi =\underset{\textbf{i}\in\mathbb{Z}^{2} }{\Sigma} \varphi_{\textbf{i}}$, where $\varphi_{\textbf{i}}$ is also a $\frac{1}{2}$ derivation of $\tilde{L(q)}$. Let $\textbf{i}\in\mathbb{Z}^{2}$, for all $\textbf{m}\in\Gamma$, we write
$$\varphi_{\textbf{i}}\left(x^{\textbf{m}}\right)=a_{\textbf{m}}x^{\textbf{m}+\textbf{i}}+b_{\textbf{m}}D\left(\textbf{m}+\textbf{i}\right),$$
$$\varphi_{\textbf{i}}\left(D\left(\textbf{m}\right)\right)=c_{\textbf{m}}x^{\textbf{m}+\textbf{i}}+d_{\textbf{m}}D\left(\textbf{m}+\textbf{i}\right).$$
Apply $\varphi_{\textbf{i}}$ to Eq.\eqref{eq:401}-\eqref{eq:403}, we have $\forall~\textbf{m},\textbf{n}\in \Gamma$,
\begin{equation}2\lambda\left(\textbf{m},\textbf{n}\right)a_{\textbf{m}+\textbf{n}}=a_{\textbf{m}}\lambda\left(\textbf{m}+\textbf{i},\textbf{n}\right)+b_{\textbf{m}}\lambda\left(\textbf{m}+\textbf{i},\textbf{n}\right)+a_{\textbf{n}}\lambda\left(\textbf{m},\textbf{n}+\textbf{i}\right)
-b_{\textbf{n}}\lambda\left(\textbf{n}+\textbf{i},\textbf{m}\right),\label{eq:411}\end{equation}
\begin{equation}2\lambda\left(\textbf{m},\textbf{n}\right)b_{\textbf{m}+\textbf{n}}D\left(\textbf{m}+\textbf{n}+\textbf{i}\right)=0,\label{eq:412}\end{equation}
\begin{equation}2\lambda\left(\textbf{m},\textbf{n}\right)a_{\textbf{m+n}}=c_{\textbf{m}}\lambda\left(\textbf{m}+\textbf{i},\textbf{n}\right)
+d_{\textbf{m}}\lambda\left(\textbf{m}+\textbf{i},\textbf{n}\right)+a_{\textbf{n}}\lambda\left(\textbf{m},\textbf{n}+\textbf{i}\right),\label{eq:413}\end{equation}
\begin{equation}2\lambda\left(\textbf{m},\textbf{n}\right)b_{\textbf{m}+\textbf{n}}=\lambda\left(\textbf{m},\textbf{n}+\textbf{i}\right)b_{\textbf{n}},\label{eq:414}\end{equation}
\begin{equation}2\lambda\left(\textbf{m},\textbf{n}\right)c_{\textbf{m}+\textbf{n}}=
-\lambda\left(\textbf{n},\textbf{m}+\textbf{i}\right)c_{\textbf{m}}+\lambda\left(\textbf{m},\textbf{n}+\textbf{i}\right)c_{\textbf{n}},\label{eq:415}\end{equation}
\begin{equation}2\lambda\left(\textbf{m},\textbf{n}\right)d_{\textbf{m}+\textbf{n}}=
\lambda\left(\textbf{m}+\textbf{i},\textbf{n}\right)d_{\textbf{m}}+\lambda\left(\textbf{m},\textbf{n}+\textbf{i}\right)d_{\textbf{n}}.\label{eq:416}\end{equation}

Since $\forall~\textbf{m}\in \Gamma \setminus\left\{-\textbf{i}\right\}$, $\exists ~\textbf{r},\textbf{s}\in \Gamma _{2}$, s.t.
     $\textbf{m}=\textbf{r}+\textbf{s}$, then by Eq.\eqref{eq:412} we have

     \begin{equation}b_{\textbf{m}}=0,~\forall~\textbf{m}\in \Gamma \setminus\left\{-\textbf{i}\right\}.\label{eq:417}\end{equation}

To determine the other coefficients, we need to consider the following cases.

\textbf{Case 1.} $\textbf{i}=\textbf{0}$. 

By substituting $b_{\textbf{m}}=0,~\forall~\textbf{m}\in \Gamma $ into Eq.\eqref{eq:411}, we have
$$\lambda\left(\textbf{m},\textbf{n}\right)\left(2a_{\textbf{m}+\textbf{n}}-a_{\textbf{m}}-a_{\textbf{n}}\right)=0,~\forall~\textbf{m},\textbf{n}\in \Gamma .$$
Then by the proof of Case 1. of Theorem \ref{E}, we know $a_{\textbf{m}}$ is a constant for all $\textbf{m}\in \Gamma $. We denote this constant as $a$.

Similarly, by Eq.\eqref{eq:415} and Eq.\eqref{eq:416}, we have
$$\lambda\left(\textbf{m},\textbf{n}\right)\left(2c_{\textbf{m}+\textbf{n}}-c_{\textbf{m}}-c_{\textbf{n}}\right)=0,~\forall~\textbf{m},\textbf{n}\in \Gamma .$$
$$\lambda\left(\textbf{m},\textbf{n}\right)\left(2d_{\textbf{m}+\textbf{n}}-d_{\textbf{m}}-d_{\textbf{n}}\right)=0,~\forall~\textbf{m},\textbf{n}\in \Gamma .$$
Then  $c_{\textbf{m}}$ and $d_{\textbf{m}}$ are constants for all $\textbf{m}\in \Gamma $. We denote the two constants as $c$ and $d$, respectively.

Then by substituting $b_{\textbf{m}}=0$, $c_{\textbf{m}}=c$ and $d_{\textbf{m}}=d$ for all ${\bf m}\in \Gamma$ into Eq.\eqref{eq:413}, we have
$$\lambda\left(\textbf{m},\textbf{n}\right)\left(a-c-d\right)=0,~\forall~\textbf{m},\textbf{n}\in \Gamma .$$
We can find a pair of $({\bf m, n})\in \Gamma^2$ such that $\lambda({\bf m, n})\neq 0,$ thus,
$$a=c+d.$$

In summary, there exist $c, d\in \mathbb C,$ such that 
\begin{center}$\forall {\bf m}\in \Gamma, \varphi_{\bf 0}(x^{\bf m})=(c+d)x^{\bf m}$ and $\varphi_{\bf 0}(D({\bf m}))=cx^{\bf m}+dD({\bf m}).$\end{center}

\textbf{Case 2.} $\textbf{i}=(i_{1},i_{2})\in \left\{0\right\} \times \mathbb{Z}^{*}$ or $\mathbb{Z}^{*}\times \left\{0\right\} $. 

Without loss of generality, we suppose $\textbf{i} \in \mathbb{Z}^{*}\times\left\{0\right\}$. By substituting $b_{\textbf{m}}=0,~\forall~\textbf{m}\in \Gamma \setminus\left\{-\textbf{i}\right\}$ into Eq.\eqref{eq:411}, we have
\begin{equation*}2\left(q^{m_{2}n_{1}}-q^{m_{1}n_{2}}\right) a_{\textbf{m}+\textbf{n}}=a_{\textbf{m}}\left(q^{m_{2}n_{1}}-q^{(m_{1}+i_{1})n_{2}}\right)+
a_{\textbf{n}}\left(q^{m_{2}(n_{1}+i_{1})}-q^{m_{1}n_{2}}\right),~\forall~\textbf{m},\textbf{n}\in\Gamma.\end{equation*}
Then by the proof of Case 2. of Theorem \ref{E}, we know $a_{\textbf{m}}=0$ for all $\textbf{m}\in \Gamma $.

Similarly, by Eq.\eqref{eq:415} and Eq.\eqref{eq:416}, we have
$$2\left(q^{m_{2}n_{1}}-q^{m_{1}n_{2}}\right) c_{\textbf{m}+\textbf{n}}=c_{\textbf{m}}\left(q^{m_{2}n_{1}}-q^{(m_{1}+i_{1})n_{2}}\right)+
c_{\textbf{n}}\left(q^{m_{2}(n_{1}+i_{1})}-q^{m_{1}n_{2}}\right),~\forall~\textbf{m},\textbf{n}\in \Gamma .$$
$$2\left(q^{m_{2}n_{1}}-q^{m_{1}n_{2}}\right) d_{\textbf{m}+\textbf{n}}=d_{\textbf{m}}\left(q^{m_{2}n_{1}}-q^{(m_{1}+i_{1})n_{2}}\right)+
d_{\textbf{n}}\left(q^{m_{2}(n_{1}+i_{1})}-q^{m_{1}n_{2}}\right),~\forall~\textbf{m},\textbf{n}\in \Gamma .$$
Then  $c_{\textbf{m}}=0$ and $d_{\textbf{m}}=0$ for all $\textbf{m}\in \Gamma $.

In summary, for ${\bf i}\in \{0\}\times \mathbb Z^*$ or $\mathbb Z^*\times \{0\}, \varphi_{\bf i}=0.$\\

\textbf{Case 3.} $\textbf{i}=(i_{1},i_{2})\in\mathbb{Z}^{*}\times\mathbb{Z}^{*}$.
 
 By substituting $b_{\textbf{m}}=0,~\forall~\textbf{m}\in \Gamma \setminus\left\{-\textbf{i}\right\}$ into Eq.\eqref{eq:411}, we have
$$2\lambda\left(\textbf{m},\textbf{n}\right)a_{\textbf{m}+\textbf{n}}=a_{\textbf{m}}\lambda\left(\textbf{m}+\textbf{i},\textbf{n}\right)+
a_{\textbf{n}}\lambda\left(\textbf{m},\textbf{n}+\textbf{i}\right),~\forall~\textbf{m},\textbf{n}\in\Gamma.$$
Then by the proof of Case 3. of Theorem \ref{E}, we know $a_{\textbf{m}}=0$ for all $\textbf{m}\in \Gamma $. Similarly, by Eq.\eqref{eq:415} and Eq.\eqref{eq:416} and by the proof of Case 3. of Theorem \ref{E}, we know $c_{\textbf{m}}=0$ and $d_{\textbf{m}}=0$ for all $\textbf{m}\in \Gamma.$

In summary, for ${\bf i}\in \mathbb Z^*\times \mathbb Z^*, \varphi_{\bf i}=0.$\\

Hence combining the analysis of the three cases above, we obtain the desired result.
\end{proof}

\begin{theorem}
There are no non-trivial transposed Poisson algebra structures defined on the algebra $\tilde{L(q)}$.
\end{theorem}
\begin{proof}
Let $\left(\tilde{L(q)},\cdot,\left[\cdot ,\cdot \right]\right)$ be a transposed Poisson algebra, then $\left( \tilde{L(q)},\cdot\right)$ is a commutative and associative algebra which satisfies Eq.\eqref{eq: 2.1}. For all ${\bf m}\in \Gamma $, we denote the left multiplication by $x^{\textbf{m}}$ or $D({\textbf m})$ in $\left( \tilde{L(q)},\cdot\right)$ as $\varphi _{x^{\textbf{m}}}$ or $\varphi _{D({\textbf{m}})}$.

Now we  consider the commutativity :

For all $\textbf{m},\textbf{n}\in\Gamma$,
$$x^{\textbf{n}}\cdot x^{\textbf{m}}=\varphi _{x^{\textbf{n}}}(x^{\textbf{m}})=(c_{x^{\textbf{n}}}+ d_{x^{\textbf{n}}})x^{\textbf{m}}, $$
$$x^{\textbf{m}}\cdot x^{\textbf{n}}=\varphi _{x^{\textbf{m}}}(x^{\textbf{n}})=(c_{x^{\textbf{m}}}+ d_{x^{\textbf{m}}})x^{\textbf{n}}.$$

$$D\left(\textbf{n}\right)\cdot D\left(\textbf{m}\right)=\varphi _{D\left(\textbf{n}\right)}(D\left(\textbf{m}\right))=
c_{D({\textbf{n}})}x^{\textbf{m}}+ d_{D({\textbf{n}})}D\left(\textbf{m}\right), $$
$$D\left(\textbf{m}\right)\cdot D\left(\textbf{n}\right)=\varphi _{D\left(\textbf{m}\right)}(D\left(\textbf{n}\right))=
c_{D({\textbf{m}})}x^{\textbf{n}}+ d_{D({\textbf{m}})}D\left(\textbf{n}\right).$$

$$ x^{\textbf{m}}\cdot D\left(\textbf{n}\right)=\varphi _{ x^{\textbf{m}}}(D\left(\textbf{n}\right))=c_{x^{\textbf{m}}}x^{\textbf{n}}+ d_{x^{\textbf{m}}}D\left(\textbf{n}\right).$$
$$ D\left(\textbf{n}\right)\cdot  x^{\textbf{m}} =\varphi _{D\left(\textbf{n}\right) }(x^{\textbf{m}})=(c_{D({\textbf{n}})}+ d_{D({\textbf{n}})})x^{\textbf{m}}.$$

Choosing $\textbf{m}\not=\textbf{n}$ in the above equations, we can deduce
$c_{x^{\textbf{m}}}=c_{ x^{\textbf{n}}}=c_{D\left(\textbf{m}\right)}=c_{D\left(\textbf{n}\right)}=0$. This leads to $x\cdot y =\varphi _{x}(y)=0$ for all $x,y\in \tilde{L(q)}$, implying that the transposed Poisson structure is trivial.
\end{proof}

\subsection{The case where $q$ is a root of unity for the $q$-Quantum Torus Lie algebra}\label{subsec4.2}
Observe that $L(q)= \underset{\textbf{m}\in \mathbb{Z}^{2} }{\oplus} \left ( L(q) \right ) _{\textbf{m}}$ is a $\mathbb{Z}^{2}$-grading, where
$\left ( L(q) \right ) _{\bf{m}}=
\mathbb{C} x^{\textbf{m}}+\mathbb CD({\bf m}),$  for $\textbf{m}\in\mathbb Z^2.$ In addition, $L(q)$ is finitely generated. 

\begin{lemma}\label{Z} \cite{bib30}
The $q$-quantum torus Lie algebra $L(q)$ is finitely generated $\mathbb{Z}^{2}$-grading Lie
algebra.
\end{lemma}

\begin{theorem} \label{H}
Let $\varphi$ be a $\frac{1}{2}$-derivation of $L(q)$. Then there exist  $a, c\in \mathbb C$, such that  
\begin{eqnarray*}\forall {\bf m}\in \Gamma, &&\varphi \left(x^{\bf m}\right)= a x^{\bf m},\\
&&\varphi \left(D\left({\bf m}\right)\right)=\left\{\begin{matrix}
  cx^{\bf m}+ aD\left({\bf m}\right), & \mbox{if}~ {\bf m}\in\Gamma_{1},\\
  aD\left({\bf m}\right),&\mbox{if}~ {\bf m}\in\Gamma_{2}.
\end{matrix}\right.
\end{eqnarray*}
\end{theorem}
\begin{proof}Let $\varphi$ be a $\frac{1}{2}$-derivation of $L(q)$, then by Lemma \ref{Z} and Lemma \ref{B}, we can write $\varphi =\underset{\textbf{i}\in\mathbb{Z}^{2} }{\Sigma} \varphi_{\textbf{i}}$, where $\varphi_{\textbf{i}}$ is also a $\frac{1}{2}$ derivation of $L(q)$. Let $\textbf{i}\in\mathbb{Z}^{2}$, for all $\textbf{m}\in\Gamma$, we write
$$\varphi_{\textbf{i}}\left(x^{\textbf{m}}\right)=a_{\textbf{m}}x^{\textbf{m}+\textbf{i}}+b_{\textbf{m}}D\left(\textbf{m}+\textbf{i}\right),$$
$$\varphi_{\textbf{i}}\left(D\left(\textbf{m}\right)\right)=c_{\textbf{m}}x^{\textbf{m}+\textbf{i}}+d_{\textbf{m}}D\left(\textbf{m}+\textbf{i}\right).$$
Applying $\varphi_{\textbf{i}}$ to Eq.\eqref{eq:11.0}-\eqref{eq:11.2}, we have $\forall~\textbf{m},\textbf{n}\in \Gamma$,
\begin{equation}2\lambda\left(\textbf{m},\textbf{n}\right)a_{\textbf{m}+\textbf{n}}=a_{\textbf{m}}\lambda\left(\textbf{m}+\textbf{i},\textbf{n}\right)+b_{\textbf{m}}h\left(\textbf{m}+\textbf{i},\textbf{n}\right)+a_{\textbf{n}}\lambda\left(\textbf{m},\textbf{n}+\textbf{i}\right)
-b_{\textbf{n}}h\left(\textbf{n}+\textbf{i},\textbf{m}\right),\label{eq:11.3}\end{equation}
\begin{equation}2\lambda\left(\textbf{m},\textbf{n}\right)b_{\textbf{m}+\textbf{n}}D\left(\textbf{m}+\textbf{n}+\textbf{i}\right)=0,\label{eq:11.4}\end{equation}
\begin{equation}2h\left(\textbf{m},\textbf{n}\right)a_{\textbf{m+n}}=c_{\textbf{m}}\lambda\left(\textbf{m}+\textbf{i},\textbf{n}\right)
+d_{\textbf{m}}h\left(\textbf{m}+\textbf{i},\textbf{n}\right)+a_{\textbf{n}}h\left(\textbf{m},\textbf{n}+\textbf{i}\right),\label{eq:11.5}\end{equation}
\begin{equation}2h\left(\textbf{m},\textbf{n}\right)b_{\textbf{m}+\textbf{n}}=g\left(\textbf{m},\textbf{n}+\textbf{i}\right)b_{\textbf{n}},\label{eq:11.6}\end{equation}
\begin{equation}2g\left(\textbf{m},\textbf{n}\right)c_{\textbf{m}+\textbf{n}}=
-h\left(\textbf{n},\textbf{m}+\textbf{i}\right)c_{\textbf{m}}+h\left(\textbf{m},\textbf{n}+\textbf{i}\right)c_{\textbf{n}},\label{eq:11.7}\end{equation}
\begin{equation}\label{eq:11.8}
2g\left(\textbf{m},\textbf{n}\right)d_{\textbf{m}+\textbf{n}}=
g\left(\textbf{m}+\textbf{i},\textbf{n}\right)d_{\textbf{m}}+g\left(\textbf{m},\textbf{n}+\textbf{i}\right)d_{\textbf{n}}.\end{equation}

Since $\forall~\textbf{m}\in \Gamma _{2}\setminus\left\{-\textbf{i}\right\}$, $\exists ~\textbf{r},\textbf{s}\in \Gamma _{2}$, s.t.
     $\textbf{m}=\textbf{r}+\textbf{s}$, then by Eq.\eqref{eq:11.4} we have

     \begin{equation}\label{eq:650}
     b_{\textbf{m}}=0,~\forall~\textbf{m}\in \Gamma _{2}\setminus\left\{-\textbf{i}\right\}.\end{equation}

To determine the other coefficients, we need to consider the following cases.

\textbf{Case 1.} $\textbf{i}=\textbf{0}$. 

By Eq.\eqref{eq:11.3} and Eq.\eqref{eq:650}, we have
$$\mathrm{det}\binom{\textbf{n}}{\textbf{m}}b_{\textbf{m}}=0, ~\forall~(\textbf{m}, \textbf{n})\in \Gamma _{1}\times \Gamma _{2}.$$
Since for all $\textbf{ m}\in \Gamma _{1}$, there exists $\textbf{n}\in \Gamma _{2}$, such that $\mathrm{det}\binom{\textbf{n}}{\textbf{m}}\not=0$, then we have
$$b_{\textbf{m}}=0,~\forall~\textbf{m}\in \Gamma _{1}.$$
Combining with Eq.\eqref{eq:650}, we get
$$b_{\textbf{m}}=0,~\forall~\textbf{m}\in \Gamma.$$

According to Eq.\eqref{eq:11.7}, we have
$$\mathrm{det}\binom{\textbf{n}}{\textbf{m}}\left(2c_{\textbf{m}+\textbf{n}}-c_{\textbf{m}}-c_{\textbf{n}}\right)=0,~\forall~\textbf{m},\textbf{n}\in \Gamma _{1}.$$
By a similar argument as in Case 1 of Theorem \ref{E} (just replace $\textbf{m}=(m_{1},m_{2}), \textbf{n}=(n_{1},n_{2}) \in \mathbb{Z}^{2}\setminus\left\{{\bf 0}\right\}$ with $\textbf{m}=(m_{1},m_{2})t, \textbf{n}=(n_{1},n_{2})t \in \Gamma _{1}$ respectively), it can be proven that $c_{\textbf{m}}$ is a constant for all $\textbf{m}\in \Gamma _{1}$. Denote this constant as $c$. Thus we proved that $$ c_{\textbf m}=c, \forall \textbf m\in\Gamma_1.$$

According to Eq.\eqref{eq:11.7}, we have
$$\lambda\left(\textbf{m},\textbf{n}\right)\left(2c_{\textbf{m}+\textbf{n}}-c_{\textbf{m}}-c_{\textbf{n}}\right)=0,~\forall~\textbf{m},\textbf{n}\in \Gamma _{2}.$$
Then by the proof of case 1. of Theorem \ref{F}, we know $c_{\textbf{m}}$ is a constant for all $\textbf{m}\in \Gamma _{2}$. Denote this constant as $c'$. 
According to Eq.\eqref{eq:11.7}, we have
$$2\mathrm{det}\binom{\textbf{n}}{\textbf{m}}c_{\textbf{m}+\textbf{n}}=\mathrm{det}\binom{\textbf{n}}{\textbf{m}}c_{\textbf{n}},~\forall~(\textbf{m}, \textbf{n})\in \Gamma _{1}\times\Gamma _{2}.$$
Since for all $\textbf{m}\in \Gamma _{2}$, $c_{\textbf{m}}=c'$, then we have
$$\mathrm{det}\binom{\textbf{n}}{\textbf{m}}c'=0,~\forall~(\textbf{m}, \textbf{n})\in \Gamma _{1}\times \Gamma _{2}.$$
We can find a pair of $\left ( \textbf{m},\textbf{n} \right ) \in \Gamma _{1}\times  \Gamma _{2}$ such that $\mathrm{det}\binom{\textbf{n}}{\textbf{m}}\not=0.$ Thus $c'=0,$ and 
$$c_{\textbf{n}}=0,~\forall~\textbf{n}\in \Gamma _{2}.$$

By Eq.\eqref{eq:11.8}, we have
$$2\lambda\left(\textbf{m},\textbf{n}\right)d_{\textbf{m}+\textbf{n}}=
\lambda\left(\textbf{m},\textbf{n}\right)d_{\textbf{m}}+\lambda\left(\textbf{m},\textbf{n}\right)d_{\textbf{n}},~\forall~\textbf{m},\textbf{n}\in \Gamma _{2}.$$
By substituting  $b_{\textbf{m}}=0$ for all $\textbf{m}\in \Gamma$ into Eq.\eqref{eq:11.3}, we have
$$2\lambda\left(\textbf{m},\textbf{n}\right)a_{\textbf{m}+\textbf{n}}=
\lambda\left(\textbf{m},\textbf{n}\right)a_{\textbf{m}}+\lambda\left(\textbf{m},\textbf{n}\right)a_{\textbf{n}},~\forall~\textbf{m},\textbf{n}\in \Gamma _{2}.$$
Thus, by a similar argument as for $c_{\textbf{m}}(\textbf{m}\in \Gamma _{2})$, it can be proven that $d_{\textbf{m}}$ and $a_{\textbf{m}}$ are constants for all $\textbf{m}\in \Gamma _{2}$. Denote the two constants as $d$ and $a$, respectively.
By substituting $c_{\textbf{m}}=0$ for all $\textbf{m}\in \Gamma _{2}$ into Eq.\eqref{eq:11.5}, we have
$$2\lambda\left(\textbf{m},\textbf{n}\right)a_{\textbf{m}+\textbf{n}}=
\lambda\left(\textbf{m},\textbf{n}\right)d_{\textbf{m}}+\lambda\left(\textbf{m},\textbf{n}\right)a_{\textbf{n}},~\forall~ \textbf{m},\textbf{n}\in \Gamma _{2}.$$
Since for all $\textbf{m}\in \Gamma _{2}$, $d_{\textbf{m}}=d$, $a_{\textbf{m}}=a$, then we have
$$\lambda\left(\textbf{m},\textbf{n}\right)\left(a-d\right)=0,~\forall~\textbf{m},\textbf{n}\in \Gamma _{2}.$$
We can find a pair of $\left(\textbf{m},\textbf{n}\right)\in \Gamma_{2}\times \Gamma_{2}$ such that $\textbf{m}+\textbf{n}\in \Gamma_{2}$, it follows that $\lambda\left(\textbf{m},\textbf{n}\right)\not=0$  and $a=d$.

By substituting $d_{\textbf{m}}=d$ for all $\textbf{m}\in \Gamma_{2}$ into Eq.\eqref{eq:11.8}, we have
$$ \mathrm{det} \binom{\textbf{n}}{\textbf{m}}\left(d-d_{\textbf{m}}\right)=0,~\forall~\textbf{m}\in \Gamma_{1}, \forall~\textbf{n}\in \Gamma _{2}.$$
For all $\textbf{m}\in \Gamma_{1}$, we can find an $\textbf{n}\in \Gamma_{2}$ such that $\mathrm{det} \binom{\textbf{n}}{\textbf{m}}\not=0$, it follows that $$\forall\textbf{m}\in \Gamma_{1}, d_{\textbf{m}}=d=a.$$

By substituting  $d_{\textbf{m}}=d$ for all $\textbf{m}\in \Gamma_{1}$ into Eq.\eqref{eq:11.5}, we get
$$ \mathrm{det} \binom{\textbf{n}}{\textbf{m}}\left(2a_{\textbf{m}+\textbf{n}}-a_{\textbf{n}}-d\right)=0,~\forall~\textbf{m},\textbf{n}\in \Gamma_{1}.$$
Thus, by a similar argument as in Case 1 of Theorem \ref{E}, it can be proven that $$\forall~\textbf{n}\in \Gamma_{1}, a_{\textbf{n}}=d=a.$$

Combining all the analysis above, we get that there exist $a,c\in \mathbb C$ such that $\forall~ {\bf m}\in \Gamma,$
$$\varphi_{\bf 0} \left(x^{\bf m}\right)= a x^{\bf m},$$
$$\varphi_{\bf 0} \left(D\left({\bf m}\right)\right)=\left\{\begin{matrix}
  cx^{\bf m}+ aD\left({\bf m}\right), & \mbox{if}~ {\bf m}\in\Gamma_{1},\\
  aD\left({\bf m}\right),&\mbox{if}~ {\bf m}\in\Gamma_{2}.
\end{matrix}\right.$$

\textbf{Case 2.} $\textbf{i}\in \Gamma _{1}$. 

By Eq.\eqref{eq:11.3} we have
$$\mathrm{det}\binom{\textbf{n}}{\textbf{m}+\textbf{i}}b_{\textbf{m}}=0, ~\forall~ ({\bf m, n})\in \Gamma_1\times \Gamma_2.$$
We know that for all $\textbf{m}\in \Gamma _{1}\setminus \left \{-\textbf{i}\right\}$, there exists $\textbf{n}\in \Gamma _{2}$ such that $\mathrm{det}\binom{\textbf{n}}{\textbf{m}+\textbf{i}}\not=0$, we have
$$b_{\textbf{m}}=0,~\forall~\textbf{m}\in\Gamma _{1}\setminus \left \{-\textbf{i}\right\}.$$
Combining with Eq.\eqref{eq:650}, we get
$$b_{\textbf{m}}=0,~\forall~\textbf{m}\in\Gamma \setminus \left \{-\textbf{i}\right\}.$$

According to Eq.\eqref{eq:11.3}, Eq.\eqref{eq:11.7} and Eq.\eqref{eq:11.8}, we have $\forall\textbf{m},\textbf{n}\in \Gamma _{2}$,
$$\lambda\left(\textbf{m},\textbf{n}\right)\left(2a_{\textbf{m}+\textbf{n}}-a_{\textbf{m}}-a_{\textbf{n}}\right)=0.$$
$$\lambda\left(\textbf{m},\textbf{n}\right)\left(2c_{\textbf{m}+\textbf{n}}-c_{\textbf{m}}-c_{\textbf{n}}\right)=0.$$
$$\lambda\left(\textbf{m},\textbf{n}\right)\left(2d_{\textbf{m}+\textbf{n}}-d_{\textbf{m}}-d_{\textbf{n}}\right)=0.$$
Then by the same arguments as in Case 1 of Theorem \ref{F}, we know that $a_{\textbf{m}}$, $c_{\textbf{m}}$, $d_{\textbf{m}}$ are constants for all
$\textbf{m}\in \Gamma _{2}$, denote them as $a$, $c$, $d$, respectively.

According to Eq.\eqref{eq:11.8},  we have
\begin{equation}\mathrm{det}\binom{\textbf{n}-\textbf{i}}{\textbf{m}}
d=\mathrm{det}\binom{\textbf{n}}{\textbf{m}+\textbf{i}}d_{\textbf{m}},~\forall~\textbf{m}\in \Gamma _{1},\textbf{n}\in \Gamma _{2}\label{eq:5.012}.\end{equation}
Particularly, by taking $\textbf{m}= - \textbf{i}$ and $\textbf{n}\in \Gamma _{2}$ such that $\mathrm{det}\binom{\textbf{i}}{\textbf{n}}\not=0$, we obtain $d=0$. 
Thus
$$d_{\textbf{m}}=d=0,~\forall~\textbf{m}\in\Gamma _{2}.$$
Then by Eq.\eqref{eq:5.012}, we get $\forall\textbf{m}\in \Gamma _{1}$, $\forall\textbf{n}\in \Gamma _{2}$, $\mathrm{det}\binom{\textbf{n}}{\textbf{m}+\textbf{i}}d_{\textbf{m}}=0$. We know that for all $\textbf{m}\in \Gamma _{1}\setminus \left \{{\bf -i}\right\}$, there exists an $\textbf{n}\in \Gamma _{2}$ such that $\mathrm{det}\binom{\textbf{n}}{\textbf{m}+\textbf{i}}\not=0.$ This shows that $\forall~\textbf{m}\in \Gamma  _{1} \setminus \left \{{\bf -i}\right\}$, $d_{\textbf{m}}=0$.

In summary,  \begin{equation}d_{\textbf{m}}=0,~\forall~\textbf{m}\in \Gamma \setminus \left \{{\bf -i}\right\}.\label{eq:5.013}\end{equation}

According to Eq.\eqref{eq:11.7}, we get
$$2\mathrm{det}\binom{\textbf{n}}{\textbf{m}}c_{\textbf{m}+\textbf{n}}=
\mathrm{det}\binom{\textbf{n}+\textbf{i}}{\textbf{m}}
c_{\textbf{n}},~\forall~\textbf{m}\in \Gamma _{1},~\forall~\textbf{n}\in \Gamma _{2}.$$
Since for all $\textbf{m}\in \Gamma _{2}$, $c_{\textbf{m}}=c$, then we have $\forall~\textbf{m}\in \Gamma _{1}$, $\forall~\textbf{n}\in \Gamma _{2}$, $\mathrm{det}\binom{\textbf{n}-\textbf{i}}{\textbf{m}}
c=0$. We can find a pair of $\left(\textbf{m},\textbf{n}\right)\in\Gamma_{1}\times\Gamma _{2}$ such that $\mathrm{det}\binom{\textbf{n}-\textbf{i}}{\textbf{m}}
\not=0$, so, $c=0$. This shows that
\begin{equation}c_{\textbf{m}}=0,~\forall~\textbf{m}\in\Gamma _{2}\label{eq:5.014}.\end{equation}

By Eq.\eqref{eq:11.7}, we have
$$2\mathrm{det}\binom{\textbf{n}}{\textbf{m}}c_{\textbf{m}+\textbf{n}}=\mathrm{det}\binom{\textbf{n}}{\textbf{m}+\textbf{i}}c_{\textbf{m}}+
\mathrm{det}\binom{\textbf{n}+\textbf{i}}{\textbf{m}}c_{\textbf{n}},~\forall~\textbf{m},\textbf{n}\in \Gamma _{1}.$$
Write $\textbf{m}=(m_{1},m_{2})t$, $\textbf{n}=(n_{1},n_{2})t$, $\textbf{i}=(i_{1},i_{2})t$ with $(m_{1},m_{2}),(n_{1},n_{2})\in\mathbb{Z}^{2} \setminus \{{\bf 0}\}$, then we have
\begin{equation}2(m_{2}n_{1}-m_{1}n_{2})c_{\textbf{m}+\textbf{n}}=((m_{2}+i_{2})n_{1}-(m_{1}+i_{1})n_{2})
c_{\textbf{m}}+(m_{2}(n_{1}+i_{1})-m_{1}(n_{2}+i_{2}))
c_{\textbf{n}}
.\label{eq:9.20}\end{equation}
We need to consider the following two cases.

\textbf{1.}$(i_{1},i_{2})\in \left\{0\right\} \times \mathbb{Z}^{*}$ or $\mathbb{Z}^{*}\times\left\{0\right\} $. 

Without loss of generality, we suppose $(i_{1},i_{2})\in \left\{0\right\} \times \mathbb{Z}^{*}$.

By taking $(n_{1},n_{2})= \textbf{e}_{1}$  and $\textbf{e}_{2}$ in Eq.\eqref{eq:9.20}, respectively, we obtain

\begin{equation}\label{eq:9.21}
-2m_{2}c_{(\textbf{m}+\textbf{e}_{1})t}=-\left(m_{2}+i_{2}\right)c_{\textbf{m}t}+\left(m_{1}i_{2}-m_{2}\right)c_{\textbf{e}_{1}t},
~\forall ~ (m_{1},m_{2})\in\mathbb{Z}^{2}\setminus\{{\bf 0}\},\end{equation}
\begin{equation}\label{eq:9.22}
2m_{1}c_{(\textbf{m}+\textbf{e}_{2})t}=m_{1}c_{\textbf{m}}+m_{1}\left(1+i_{2}\right)c_{\textbf{e}_{2}t},~\forall~  (m_{1},m_{2})\in\mathbb{Z}^{2}\setminus\{{\bf 0}\}.\end{equation}
By taking $m_{2}= 0$ in Eq.\eqref{eq:9.21}, we get
$$i_{2}\left(c_{(m_{1},0)t}-m_{1}c_{\textbf{e}_{1}t}\right)=0,~\forall~m_{1}\in\mathbb{Z}^{*}.$$
Since $i_{2}\not=0$, we get
\begin{equation}c_{(m_{1},0)t}=m_{1}c_{\textbf{e}_{1}t},~\forall~m_{1}\in\mathbb{Z}^{*}.\label{eq:9.24}\end{equation}
By Eq.\eqref{eq:9.22}, we have
\begin{equation}2c_{(m_{1},m_{2}+1)t}=c_{\textbf{m}}+\left(1+i_{2}\right)c_{\textbf{e}_{2}t},~\forall ~ (m_{1},m_{2})\in\mathbb{Z}^{*}\times\mathbb{Z}.\label{eq:9.23}\end{equation}
For a fixed $m_{1}\in\mathbb{Z}^{*}$, treat $\left(c_{\textbf{m}}-\left(1+i_{2}\right)c_{\textbf{e}_{2}t}\right)_{m_{2}\in\mathbb{Z}}$ as a geometric sequence, then by Eq.\eqref{eq:9.23}, we have
\begin{equation*}c_{\textbf{m}}=\left( c_{(m_{1},0)t}-\left(1+i_{2}\right)c_{\textbf{e}_{2}t}\right)\left(\frac{1}{2}\right)^{m_{2}}+\left(1+i_{2}\right)c_{\textbf{e}_{2}t},
~\forall~ \textbf{m}\in\mathbb{Z}^{*}\times\mathbb{Z}.\label{eq:9.25} \end{equation*}
By substituting Eq.\eqref{eq:9.24} into the above equation, we get
\begin{equation}c_{\textbf{m}}=m_{1}\left(\frac{1}{2}\right)^{m_{2}}c_{\textbf{e}_{1}t}+\left(1+i_{2}\right)\left(1-\left(\frac{1}{2}\right)^{m_{2}}\right)c_{\textbf{e}_{2}t},
~\forall~  (m_{1},m_{2})\in\mathbb{Z}^{*}\times\mathbb{Z}.
\label{eq:9.26}\end{equation}
By substituting Eq.\eqref{eq:9.26} into Eq.\eqref{eq:9.20}, for those $({\bf m, n})\in (t\mathbb{Z})^{*}\times t\mathbb Z$,  such that ${\bf m+n}\in (t\mathbb{Z})^{*}\times t\mathbb Z$, we have for all $(m_{2},n_{2})\in\mathbb{Z}^{2}$,
\begin{eqnarray}\label{eq:9.27}
&&2\mathrm{det}\binom{\textbf{n}}{\textbf{m}}\left(\left(m_{1}+n_{1}\right)\left(\frac{1}{2}\right)^{m_{2}+n_{2}} c_{\textbf{e}_{1}t}+
\left(1+i_{2}\right)\left(1-\left(\frac{1}{2}\right)^{m_{2}+n_{2}}\right) c_{\textbf{e}_{2}t}\right)\\\nonumber
&=&
\mathrm{det}\binom{\textbf{n}}{\textbf{m}+\textbf{i}}\left(m_{1} \left(\frac{1}{2}\right)^{m_{2}}c_{\textbf{e}_{1}t}+\left(1+i_{2}\right)\left(1-\left(\frac{1}{2}\right)^{m_{2}}\right) c_{\textbf{e}_{2}t}\right)
\\ \nonumber
&&+\mathrm{det}\binom{\textbf{n}+\textbf{i}}{\textbf{m}}\left(n_{1}\left(\frac{1}{2}\right)^{n_{2}}c_{\textbf{e}_{1}t}+\left(1+i_{2}\right)\left(1-\left(\frac{1}{2}\right)^{n_{2}}\right) c_{\textbf{e}_{2}t}\right).\nonumber
\end{eqnarray}

Now we need to consider the two subcases: $i_{2}= -1$ and $i_{2}\ne -1$.

\textbf{Subcase (i).} $i_{2}= -1$. 

By Eq.\eqref{eq:9.27}, we have
\begin{equation}\label{eq:9.28}
(2\mathrm{det}\binom{\textbf{n}}{\textbf{m}}\left(m_{1}+n_{1}\right)\left(\frac{1}{2}\right)^{m_{2}+n_{2}}-
\mathrm{det}\binom{\textbf{n}}{\textbf{m}+\textbf{i}}m_{1}\left(\frac{1}{2}\right)^{m_{2}}-
\mathrm{det}\binom{\textbf{n}+\textbf{i}}{\textbf{m}}n_{1}\left(\frac{1}{2}\right)^{n_{2}})c_{\textbf{e}_{1}t}=0.\end{equation}
Since $m_{1} \not= 0,$ $n_{1} \not= 0$ and  $m_{1}+n_{1} \not= 0,$ so by setting $(m_{1},m_{2})=(2,2)$, $(n_{1},n_{2})=(1,1)$ in Eq\eqref{eq:9.28} , we have
$-\frac{1}{2}c_{\textbf{e}_{1}t}=0,$
then
$c_{\textbf{e}_{1}t}=0.$
By taking $i_{2}= -1$ in Eq.\eqref{eq:9.26} and then substituting $c_{\textbf{e}_{1}t}=0$ into Eq.\eqref{eq:9.26}, we get
$$c_{\textbf{m}}=0,~\forall  ~\textbf{m}\in(t\mathbb{Z})^{*}\times t\mathbb{Z}.$$
By taking $m_{1}= 1$, $(n_{1},n_{2})=(-1,0)$ in Eq\eqref{eq:9.20}, we get
$$-2m_{2}c_{(0,m_{2})t}=0,~\forall  ~m_{2}\in\mathbb{Z}$$
then
$$c_{(0,m_{2})t}=0,~\forall  ~m_{2}\in\mathbb{Z}^{*}.$$
Thus
$$c_{\textbf{m}}=0,~\forall~ \textbf{m}\in\Gamma_{1}.$$

\textbf{Subcase (ii).}  $i_{2}\not= -1$. 

By taking $(m_{1},m_{2})=(1,1)$, $(n_{1},n_{2})=(2,2)$ and $(m_{1},m_{2})=(1,1)$, $(n_{1},n_{2})=(-2,-2)$ into Eq\eqref{eq:9.27}, respectively, we have
$$\left\{\begin{matrix}
2c_{\textbf{e}_{1}t}+\left(1+i_{2}\right)c_{\textbf{e}_{2}t}=0,
 \\\frac{7}{2}c_{\textbf{e}_{1}t}+\left(1+i_{2}\right)c_{\textbf{e}_{2}t}=0.
\end{matrix}\right.$$
Solving the system of equations above, we get
\begin{equation*}c_{\textbf{e}_{1}t}=c_{\textbf{e}_{2}t}=0.\label{eq:9.93}\end{equation*}
By substituting $c_{\textbf{e}_{1}t}=0$ and $c_{\textbf{e}_{2}t}=0$ into Eq\eqref{eq:9.26}, we get
$$c_{\textbf{m}}=0,~\forall~ \textbf{m}\in (t\mathbb{Z})^{*}\times t\mathbb{Z}.$$
Then, by taking $m_{1} = 1$, $(n_{1},n_{2})=(-1,0)$ in Eq\eqref{eq:9.20}, we get
$$-2m_{2}c_{(0,m_{2})t}=0,~\forall  ~m_{2}\in\mathbb{Z},$$
then
$$c_{(0,m_{2})t}=0,~\forall  ~m_{2}\in\mathbb{Z}^{*}.$$
Thus
$$c_{\textbf{m}}=0,~\forall~ \textbf{m}\in\Gamma_{1}.$$

\textbf{2.} $(i_{1},i_{2})\in \mathbb{Z}^{*}\times \mathbb{Z}^{*}$. 

On the one hand, by setting $(n_{1},n_{2})=(1,0)$, $m_{2}= 0$ in Eq\eqref{eq:9.20}, we have
$$i_{2}\left(c_{(m_{1},0)t}-m_{1}c_{\textbf{e}_{1}t}\right)=0,~\forall  ~m_{1}\in\mathbb{Z}^{*}.$$
Since $i_{2}\not=0$, then we have
\begin{equation}c_{(m_{1},0)t}=m_{1}c_{\textbf{e}_{1}t},~ \forall  ~m_{1}\in\mathbb{Z}^{*}\label{eq:9.33}.\end{equation}
On the other hand, by taking $(n_{1},n_{2})=(0,1)$, $m_{1}= 0$ in Eq\eqref{eq:9.20}, we have
$$i_{1}\left(c_{(0,m_{2})t}-m_{2}c_{\textbf{e}_{2}t}\right)=0,~\forall  ~m_{2}\in\mathbb{Z}^{*}$$
Since $i_{1}\not=0$, then we have
\begin{equation}c_{(0,m_{2})t}=m_{2}c_{\textbf{e}_{2}t},~\forall  ~m_{2}\in\mathbb{Z}^{*}\label{eq:9.34}.\end{equation}
By taking  $n_{1} = 0$, $m_{2} = 0$ in Eq\eqref{eq:9.20}, we get
\begin{equation}2m_{1}n_{2}c_{(m_{1},n_{2})t}=n_{2}\left(m_{1}+i_{1}\right)c_{(m_{1},0)t}+m_{1} \left(n_{2}+i_{2} \right)c_{(0,n_{2})t},~\forall  ~m_{1},n_{2}\in\mathbb{Z}^{*}\label{eq:9.35}.\end{equation}
By substituting Eq.\eqref{eq:9.33} and Eq.\eqref{eq:9.34} into Eq.\eqref{eq:9.35}, we have
\begin{equation}2m_{1}n_{2}c_{(m_{1},n_{2})t}=n_{2}\left(m_{1}+i_{1}\right)m_{1}c_{\textbf{e}_{1}t}+m_{1}\left(n_2+i_{2}\right)n_{2}c_{\textbf{e}_{2}t},~\forall  ~m_{1},n_{2}\in\mathbb{Z}^{*}.\end{equation}
For all $m_{1}\in\mathbb{Z}^{*}$, $n_{2}\in\mathbb{Z}^{*}$, since $m_{1}n_{2}\not=0$, then
\begin{equation}2c_{(m_{1},n_{2})t}=\left(m_{1}+i_{1}\right)c_{\textbf{e}_{1}t}+\left(n_{2}+i_{2}\right)c_{\textbf{e}_{2}t},~\forall  ~m_{1},n_{2}\in\mathbb{Z}^{*}.\label{eq:9.36}\end{equation}
By substituting Eq.\eqref{eq:9.36} into Eq.\eqref{eq:9.20}, for those $\textbf{m},\textbf{n}\in(t\mathbb{Z})^{*}\times(t\mathbb{Z})^{*}$ such that ${\bf m+n}\in (t\mathbb{Z})^{*}\times(t\mathbb{Z})^{*}$, we have
\begin{eqnarray*}
&&2\mathrm{det}\binom{\textbf{n}}{\textbf{m}}\left(\left(m_{1}+n_{1}+i_{1}\right) c_{\textbf{e}_{1}t}+\left(m_{2}+n_{2}+i_{2}\right)c_{\textbf{e}_{2}t}\right)
\\
&=&\mathrm{det}\binom{\textbf{n}}{\textbf{m}+\textbf{i}}\left(\left(m_{1}+i_{1}\right)c_{\textbf{e}_{1}t}+\left(m_{2}+i_{2}\right)c_{\textbf{e}_{2}t}\right)
\\
&&+\mathrm{det}\binom{\textbf{n}+\textbf{i}}{\textbf{m}}\left(\left(n_{1}+i_{1}\right)c_{\textbf{e}_{1}t}+\left(n_{2}+i_{2}\right)c_{\textbf{e}_{2}t}\right),\end{eqnarray*}
Particularly, by setting $(m_{1},m_{2})=(-2i_{1},-i_{2})$, $(n_{1},n_{2})=(-i_{1},-i_{2})$ and $(m_{1},m_{2})=(-3i_{1},-i_{2})$, $(n_{1},n_{2})=(-i_{1},-i_{2})$ respectively
into the above equation, we have
$$\left\{\begin{matrix}
3i_{1}c_{\textbf{e}_{1}t}+2i_{2}c_{\textbf{e}_{2}t}=0, \\
4i_{1}c_{\textbf{e}_{1}t}+2i_{2}c_{\textbf{e}_{2}t}=0.
\end{matrix}\right.$$
Solving the system of equations above yields
$$c_{\textbf{e}_{1}t}=c_{\textbf{e}_{2}t}=0.$$
By substituting $c_{\textbf{e}_{1}t}=0$ and $c_{\textbf{e}_{2}t}=0$ into Eq.\eqref{eq:9.36}, we get
$$c_{(m_{1},n_{2})t}=0,~\forall ~(m_{1},n_{2})\in\mathbb{Z}^{*}\times \mathbb{Z}^{*}.$$
i.e.
\begin{equation*}c_{\textbf{m}}=0,~\forall~ \textbf{m}\in(t\mathbb{Z})^{*}\times (t\mathbb{Z})^{*}.\end{equation*}
By substituting $c_{\textbf{e}_{1}t}=0$ into Eq.\eqref{eq:9.33}, we get
\begin{equation*}c_{\bf m}=0,~\forall~ {\bf m}\in(t\mathbb{Z})^*\times \{0\}.\end{equation*}
By substituting $c_{\textbf{e}_{2}t}=0$ into Eq.\eqref{eq:9.34}, we get
\begin{equation*}
c_{\bf m}=0,~\forall~ {\bf m}\in\{0\}\times(t\mathbb{Z})^*.\end{equation*}
Thus,
$$c_{\textbf{m}}=0,~\forall~\textbf{m}\in \Gamma_{1}.$$

By substituting Eq.\eqref{eq:5.013} and Eq.\eqref{eq:5.014} into Eq.\eqref{eq:11.5}, we have
$$2\lambda\left(\textbf{m},\textbf{n}\right)a_{\textbf{m}+\textbf{n}}=\lambda\left(\textbf{m},\textbf{n}\right)a_{\textbf{n}},~\forall~\textbf{m},\textbf{n}\in\Gamma _{2}.$$
We know that for all $\textbf{n}\in \Gamma _{2}$, there exists $\textbf{m}\in \Gamma _{2}$ such that $\textbf{m}+\textbf{n}\in \Gamma _{2}$, then $\lambda\left(\textbf{m},\textbf{n}\right)\not=0$. In addition, for all $\textbf{n}\in \Gamma _{2}$, $a_{\textbf{n}}=a=0$, it follows that $a=0$. Thus
$$a_{\textbf{m}}=0,~\forall~\textbf{m}\in\Gamma _{2}.$$

Now, by Eq.\eqref{eq:11.5}, we have
$$2\mathrm{det}\binom{\textbf{n}}{\textbf{m}}
a_{\textbf{m}+\textbf{n}}=\mathrm{det}\binom{\textbf{n}+\textbf{i}}{\textbf{m}}
a_{\textbf{n}},~\forall~\textbf{m},\textbf{n}\in\Gamma_{1}.$$
We write $\textbf{m}=(m_1, m_2)t$, $\textbf{n}=(n_1, n_2)t$, $\textbf{i}=(i_1,i_2)t$ with $(m_1, m_2), (n_1, n_2)\in\mathbb{Z}^{2} \setminus \{{\bf 0}\}$, then we have
\begin{equation}\label{eq:101}
2\left(m_{2}n_{1}-m_{1}n_{2}\right)a_{\textbf{m+n}}=
\left(m_{2}\left(n_{1}+i_{1}\right)-m_{1}\left(n_{2}+i_{2}\right)\right)a_{\textbf{n}},~\forall~(m_1, m_2), (n_1, n_2)\in\mathbb Z^2\setminus\{{\bf 0}\}.\end{equation}
Since $\left(i_{1},i_{2}\right)\not=\left(0,0\right)$, without loss of generality, we suppose $i_{2}\not=0$. By taking $(m_1, m_2)=(1,0)$ and $m_2=0$ in Eq\eqref{eq:101}, respectively, we have
\begin{equation}\label{eq:102}
-2n_{2}a_{(n_{1}+1,n_{2})t}=-\left ( n_{2}+i_{2} \right ) a_{(n_{1},n_{2})t}, ~\forall ~(n_1, n_2)\in \mathbb Z^2\setminus\{{\bf 0}\}\end{equation}
\begin{equation}\label{eq:103}
-2n_{2}a_{(n_{1}+m_1,n_{2})t}=-\left ( n_{2}+i_{2} \right ) a_{(n_{1},n_{2})t}, ~\forall ~(n_1, n_2)\in \mathbb Z^2\setminus\{{\bf 0}\}, \forall~ m_1\in \mathbb Z^*.\end{equation}
Then, by Eq.\eqref{eq:102}, we have 
 \begin{equation}\label{eq:4.41}
     a_{(n_{1},0)t}=0, ~\forall ~n_1\in\mathbb Z^*.
 \end{equation}
By Eq.\eqref{eq:102} and Eq.\eqref{eq:103}, we have
\begin{equation}\label{eq:104}
a_{(n_{1}+1,n_{2})t}=a_{(n_{1}+m_1,n_{2})t}, ~\forall ~(n_1,n_2)\in \mathbb Z\times \mathbb Z^*, \forall ~m_1\in \mathbb Z^*.\end{equation}
By Eq.\eqref{eq:102} and Eq.\eqref{eq:104}, we have
$$\left ( n_{2}-i_{2} \right )a_{(n_1,n_{2})t}=0, ~\forall (n_1, n_2)\in \mathbb Z\times\mathbb Z^*.$$
It follows that \begin{equation}\label{eq:4.43}
a_{(n_1, n_{2})t}=0, ~\forall ~(n_1, n_2)\in \mathbb Z\times (\mathbb Z^*\setminus\{i_2\}).\end{equation}
Particularly,  $$\forall n_1\in\mathbb Z,~ a_{(n_1, i_2+1)t}=0.$$ 
By taking $(m_1, m_2)=(0,1), n_2=i_2$ and $m_1=0, n_2=i_2$ in Eq.\eqref{eq:101}, respectively,
 we have
\begin{equation*}
2n_{1}a_{(n_{1}, i_{2}+1)t}=\left ( n_{1}+i_{1} \right ) a_{(n_{1}, i_{2})t}, ~\forall~n_1\in\mathbb Z.
\end{equation*}
\begin{equation*}
    2n_1a_{(n_1, m_2+i_2)t}=(n_1+i_1)a_{(n_1, i_2)t}, ~\forall ~ m_2\in\mathbb Z^*, ~\forall~ n_1\in\mathbb Z
\end{equation*}
It follows that \begin{equation}\label{eq:4.44}
    a_{(n_1, i_2)t}=0, ~\forall ~n_1\in\mathbb Z\setminus\{-i_1\}~ \mbox{and} ~ a_{(-i_{1}, i_{2})t}=a_{(-i_{1}, i_{2}+1)t}=0.
\end{equation}
Combining Eq.\eqref{eq:4.41}, Eq.\eqref{eq:4.43} and Eq.\eqref{eq:4.44}, we obtain $a_{\bf n}=0, ~\forall~ n\in\Gamma_1.$

In summary,
$$a_{\textbf{n}}=0,~\forall~\textbf{n}\in \Gamma.$$

Combining all the results above, we proved that for all $\textbf{i}\in \Gamma _{1}$, $\varphi_{\textbf{i}}=0$.

\textbf{Case 3.} $\textbf{i}\in \Gamma _{2}$. 

By Eq.\eqref{eq:11.6}, we have
$$2\lambda\left(\textbf{m},\textbf{n}\right)b_{\textbf{m}+\textbf{n}}=\lambda\left(\textbf{m},\textbf{n}+\textbf{i}\right)b_{\textbf{n}},~\forall~\textbf{n}\in \Gamma _{1},\textbf{m}\in \Gamma _{2}.$$
Since for ${\bf m}\in \Gamma_1, {\bf n}\in \Gamma_2,$ $\textbf{m}+\textbf{n}\in \Gamma _{2}$, then by Eq.\eqref{eq:650}, we know $b_{\textbf{m+n}}=0.$ It follows that for all $\textbf{n}\in \Gamma _{1}$, there exists $\textbf{m}\in \Gamma _{2}$ such that $\textbf{m}+\textbf{n}+\textbf{i}\in \Gamma _{2}$, so that $\lambda\left(\textbf{m},\textbf{n}+\textbf{i}\right)\not=0.$ This shows that
$$b_{\textbf{n}}=0,~\forall~\textbf{n}\in \Gamma _{1}.$$
Combining with Eq.\eqref{eq:650}, we have
$$b_{\textbf{m}}=0,~\forall~\textbf{m}\in \Gamma\setminus\left\{-\textbf{i}\right\}.$$

Substituting $b_{\textbf{m}}=0$ for all ${\bf m}\in \Gamma\setminus\{{\bf -i}\}$ into Eq.\eqref{eq:11.3}, we have
$$2\lambda\left(\textbf{m},\textbf{n}\right)a_{\textbf{m}+\textbf{n}}=
\lambda\left(\textbf{m}+\textbf{i},\textbf{n}\right)a_{\textbf{m}}+\lambda\left(\textbf{m},\textbf{n}+\textbf{i}\right)a_{\textbf{n}},
~\forall~\textbf{m},\textbf{n}\in\Gamma .$$
Then by the proof of Case 2 and Case 3 of Theorem \ref{F}, we get
$$a_{\textbf{m}}=0,~\forall~\textbf{m}\in \Gamma.$$

By Eq.\eqref{eq:11.7}, we have
$$2\lambda\left(\textbf{m},\textbf{n}\right)c_{\textbf{m}+\textbf{n}}=
\lambda\left(\textbf{m}+\textbf{i},\textbf{n}\right)c_{\textbf{m}}+
\lambda\left(\textbf{m},\textbf{n}+\textbf{i}\right)c_{\textbf{n}},~\forall~\textbf{m},\textbf{n}\in \Gamma _{2},$$
Then by a similar argument as in Case 2 and Case 3 of Theorem \ref{F}, we can prove
$$c_{\textbf{m}}=0,~\forall~\textbf{m}\in \Gamma_{2}.$$
Substituting $a_{\textbf{m}}=0,~\forall~\textbf{m}\in \Gamma$ and $c_{\textbf{m}}=0,~\forall~\textbf{m}\in \Gamma_{2}$ into Eq.\eqref{eq:11.5}, we have
$$d_{\textbf{m}}h\left(\textbf{m}+\textbf{i},\textbf{n}\right)=0, ~\forall~{\bf (m, n)}\in \Gamma_2\times\Gamma.$$
There exist an $\textbf{n}\in \Gamma$ such that $h\left(\textbf{m}+\textbf{i}, {\bf n}\right)\not=0$, then
$$d_{\textbf{m}}=0,~\forall~\textbf{m}\in \Gamma_{2}.$$

By Eq.\eqref{eq:11.7}, we have
$$2\mathrm{det} \binom{\textbf{n}}{\textbf{m}}c_{\textbf{m}+\textbf{n}}=
\mathrm{det} \binom{\textbf{n}}{\textbf{m}+\textbf{i}}c_{\textbf{m}}+
\mathrm{det} \binom{\textbf{n}+\textbf{i}}{\textbf{m}}c_{\textbf{n}},\forall\textbf{m},\textbf{n}\in \Gamma _{1},$$
Write $\textbf{m}=(m_{1},m_{2})t$, $\textbf{n}=(n_{1},n_{2})t$, $\textbf{i}=(i_1, i_2)$ with $(m_{1},m_{2}),(n_{1},n_{2})\in\mathbb{Z}^{2} \setminus\{{\bf 0}\}$, then we have
\begin{equation}\label{eq:9.2011}
2(m_{2}n_{1}-m_{1}n_{2})c_{\textbf{m}+\textbf{n}}=((m_{2}+\frac{i_{2}}{t})n_{1}-(m_{1}+\frac{i_{1}}{t})n_{2})
c_{\textbf{m}}+(m_{2}(n_{1}+\frac{i_{1}}{t})-m_{1}(n_{2}+\frac{i_{2}}{t}))
c_{\textbf{n}}
.\end{equation}
Then by a similar argument as $c_{\bf m}=0$ for ${\bf m}\in\Gamma_1$ in Case 2, we can prove $c_{\bf m}=0, ~\forall~ {\bf m}\in\Gamma_1.$  
By \eqref{eq:11.8}, we have
$$2\mathrm{det} \binom{\textbf{n}}{\textbf{m}}d_{\textbf{m}+\textbf{n}}=
\mathrm{det} \binom{\textbf{n}}{\textbf{m}+\textbf{i}}d_{\textbf{m}}+
\mathrm{det} \binom{\textbf{n}+\textbf{i}}{\textbf{m}}d_{\textbf{n}},~\forall~\textbf{m},\textbf{n}\in \Gamma _{1}.$$
Thus, by a similar argument as $c_{\textbf{m}}=0$ for ${\bf m}\in\Gamma_1$, we get
$$d_{\textbf{m}}=0,~\forall~\textbf{m}\in \Gamma_{1}.$$

Combining all the results above, we proved that for all $\textbf{i}\in \Gamma _{2}$, $\varphi_{\textbf{i}}=0$.

Hence combining the analysis of the three cases above, we obtain the desired result.
\end{proof}

\begin{theorem}
There are no non-trivial transposed Poisson algebra structures defined on the algebra $L(q)$.
\end{theorem}

\begin{proof}
Let $\left( L(q),\cdot,\left[\cdot ,\cdot \right]\right)$ be a transposed Poisson algebra, then $\left( L(q),\cdot\right)$ is a commutative and associative algebra which satisfies Eq.\eqref{eq: 2.1}. For all ${\bf m}\in \Gamma$, we denote the left multiplication by $x^{\textbf{m}}$ or $D({\textbf m})$ in $\left( L(q),\cdot\right)$ as $\varphi _{x^{\textbf{m}}}$ or $\varphi _{D({\textbf m})}$.

Now we consider the commutativity:

 According to Theorem \ref{H}, for  all $\textbf{m},\textbf{n}\in\Gamma$,
$$x^{\textbf{n}}\cdot x^{\textbf{m}}=\varphi _{x^{\textbf{n}}}(x^{\textbf{m}})=a_{x^{\textbf{n}}}x^{\textbf{m}}, $$
$$x^{\textbf{m}}\cdot x^{\textbf{n}}=\varphi _{x^{\textbf{m}}}(x^{\textbf{n}})=a_{x^{\textbf{m}}}x^{\textbf{n}}.$$
For all $\textbf{m},\textbf{n}\in\Gamma_{1}$:
$$D\left(\textbf{n}\right)\cdot D\left(\textbf{m}\right)=\varphi _{D\left(\textbf{n}\right)}(D\left(\textbf{m}\right))=c_{D\left(\textbf{n}\right)}x^{\textbf{m}}+a_{D\left(\textbf{n}\right)}D\left(\textbf{m}\right), $$
$$D\left(\textbf{m}\right)\cdot D\left(\textbf{n}\right)=\varphi _{D\left(\textbf{m}\right)}(D\left(\textbf{n}\right))=c_{D\left(\textbf{m}\right)}x^{\textbf{n}}+a_{D\left(\textbf{m}\right)}D\left(\textbf{n}\right).$$
For all $\textbf{m},\textbf{n}\in\Gamma_{2}$:
$$D\left(\textbf{n}\right)\cdot D\left(\textbf{m}\right)=\varphi _{D\left(\textbf{n}\right)}(D\left(\textbf{m}\right))=a_{D\left(\textbf{n}\right)}D\left(\textbf{m}\right), $$
$$D\left(\textbf{m}\right)\cdot D\left(\textbf{n}\right)=\varphi _{D\left(\textbf{m}\right)}(D\left(\textbf{n}\right))=a_{D\left(\textbf{m}\right)}D\left(\textbf{n}\right).$$
For all $\textbf{m}\in\Gamma_{1},\textbf{n}\in\Gamma_{2}$:
$$D\left(\textbf{n}\right)\cdot D\left(\textbf{m}\right)=\varphi _{D\left(\textbf{n}\right)}(D\left(\textbf{m}\right))=c_{D\left(\textbf{n}\right)}x^{\textbf{m}}+a_{D\left(\textbf{n}\right)}D\left(\textbf{m}\right), $$
$$D\left(\textbf{m}\right)\cdot D\left(\textbf{n}\right)=\varphi _{D\left(\textbf{m}\right)}(D\left(\textbf{n}\right))=a_{D\left(\textbf{m}\right)}D\left(\textbf{n}\right).$$
For all $\textbf{m}\in\Gamma,\textbf{n}\in\Gamma_{1}$:
$$D\left(\textbf{n}\right)\cdot x^{\textbf{m}}=\varphi _{D\left(\textbf{n}\right)}(x^{\textbf{m}})=a_{D\left(\textbf{n}\right)}x^{\textbf{m}} ,$$
$$ x^{\textbf{m}}\cdot D\left(\textbf{n}\right)=\varphi _{ x^{\textbf{m}}}(D\left(\textbf{n}\right))=c_{ x^{\textbf{m}}}x^{\textbf{n}}+a_{ x^{\textbf{m}}}D\left(\textbf{n}\right).$$
For all $\textbf{m}\in\Gamma,\textbf{n}\in\Gamma_{2}$:
$$D\left(\textbf{n}\right)\cdot x^{\textbf{m}}=\varphi _{D\left(\textbf{n}\right)}(x^{\textbf{m}})=a_{D\left(\textbf{n}\right)}x^{\textbf{m}} ,$$
$$ x^{\textbf{m}}\cdot D\left(\textbf{n}\right)=\varphi _{ x^{\textbf{m}}}(D\left(\textbf{n}\right))=a_{ x^{\textbf{m}}}D\left(\textbf{n}\right).$$

Choosing $\textbf{m}\not=\textbf{n}$ in the above equations, we can deduce
$a_{x^{\textbf{m}}}=a_{ x^{\textbf{n}}}=a_{D\left(\textbf{m}\right)}=a_{D\left(\textbf{n}\right)}=c_{D\left(\textbf{m}\right)}=c_{D\left(\textbf{n}\right)}=c_{ x^{\textbf{m}}}=0$. This leads to $x\cdot y =\varphi _{x}(y)=0,$  for all $x,y\in L(q)$, implying that the transposed Poisson structure on $L(q)$ is trivial.
\end{proof}

\backmatter





\bmhead{Acknowledgements}

The work is supported by the NSFC (12201624).

\bibliography{sn-bibliography}

\end{document}